\documentclass[letterpaper, 11pt,  reqno]{amsart}

\usepackage{amsmath,amssymb,amscd,amsthm,amsxtra, esint}
\usepackage{hyperref}

\allowdisplaybreaks[2]

\newtheorem{theorem}{Theorem} [section]

\newtheorem{lemma}[theorem]{Lemma}
\newtheorem{proposition}[theorem]{Proposition}
\newtheorem{remark}[theorem]{Remark}

\DeclareMathOperator*{\supp}{supp}

\newcommand{\I}{\hspace{0.5mm}\text{I}\hspace{0.5mm}}
\newcommand{\II}{\text{I \hspace{-2.8mm} I} }

\newcommand{\noi}{\noindent}
\newcommand{\Z}{\mathbb{Z}}
\newcommand{\R}{\mathbb{R}}

\newcommand{\T}{\mathbb{T}}

\let\Re=\undefined\DeclareMathOperator*{\Re}{Re}
\let\Im=\undefined\DeclareMathOperator*{\Im}{Im}

\let\P= \undefined
\newcommand{\P}{\mathbf{P}}

\newcommand{\E}{\mathbb{E}}

\newcommand{\F}{\mathcal{F}}

\newcommand{\al}{\alpha}
\newcommand{\be}{\beta}
\newcommand{\dl}{\delta}

\newcommand{\nb}{\nabla}

\newcommand{\Dl}{\Delta}
\newcommand{\eps}{\varepsilon}

\newcommand{\g}{\gamma}

\newcommand{\ld}{\lambda}

\newcommand{\Si}{\Sigma}
\newcommand{\ft}{\widehat}

\newcommand{\wt}{\widetilde}
\newcommand{\cj}{\overline}

\newcommand{\dt}{\partial_t}

\renewcommand{\l}{\ell}
\renewcommand{\o}{\omega}
\renewcommand{\O}{\Omega}

\newcommand{\les}{\lesssim}
\newcommand{\ges}{\gtrsim}

\newcommand{\jb}[1]
{\langle #1 \rangle}

\renewcommand{\S}{\mathcal{S}}

\newcommand{\pa}{\partial}

\newtheorem*{ackno}{Acknowledgement}

\newcommand{\bu}{{\bf u}}
\newcommand{\bz}{{\bf z}}
\newcommand{\bv}{{\bf v}}
\newcommand{\bZ}{{\bf Z}}

\numberwithin{equation}{section}
\numberwithin{theorem}{section}

\newcommand{\et}{\eta_{_T}}
\newcommand{\fet}{\wt \eta_{_T}}
\renewcommand{\H}{\mathcal{H}}

\newsavebox{\foobox}
\newcommand{\slantbox}[2][0]{\mbox{%
        \sbox{\foobox}{#2}%
        \hskip\wd\foobox
        \pdfsave
        \pdfsetmatrix{1 0 #1 1}%
        \llap{\usebox{\foobox}}%
        \pdfrestore
}}
\newcommand\unslant[2][-.25]{\slantbox[#1]{$#2$}}

\newcommand{\pphi}{\hspace{1pt}\pmb{\unslant\phi}}

\begin{document}
\baselineskip = 14pt


\title[Almost sure  GWP of the periodic energy-critical NLW]
{A remark on almost sure global well-posedness of the energy-critical
defocusing  nonlinear wave equations
in the periodic setting}

\author[T.~Oh and O.~Pocovnicu]
{Tadahiro Oh and   Oana Pocovnicu}

\address{
Tadahiro Oh\\
School of Mathematics\\
The University of Edinburgh\\
and The Maxwell Institute for the Mathematical Sciences\\
James Clerk Maxwell Building\\
The King's Buildings\\
Peter Guthrie Tait Road\\
Edinburgh\\ 
EH9 3FD\\
 United Kingdom}
 
\email{hiro.oh@ed.ac.uk}

\address{
Oana Pocovnicu\\
Department of Mathematics\\
Princeton University\\
Fine Hall\\
Washington Rd.\\
Princeton\\ NJ 08544-1000\\USA\\
and 
Department of Mathematics\\
Heriot-Watt University
and The Maxwell Institute for the Mathematical Sciences\\
Edinburgh\\
EH14 4AS\\United Kingdom}

\email{opocovnicu@math.princeton.edu}



\begin{abstract}
In this note, 
we prove almost sure  global well-posedness of the energy-critical defocusing nonlinear wave equation on $\T^d$,
$d = 3, 4,$ and $ 5$, 
with random initial data below the energy space.

\end{abstract}

\maketitle

%

\date{\today}

%
%

\baselineskip = 15pt

\section{Introduction}\label{SEC:intro}

\subsection{Energy-critical nonlinear wave equations}

We consider the Cauchy problem for the energy-critical defocusing nonlinear wave equation (NLW) on 
the $d$-dimensional torus $\T^d =(\R/2\pi\Z)^d$, $d=3, 4$ or $5$:
\begin{equation}\label{NLW}
\begin{cases}
\pa_{t}^2 u-\Delta u+|u|^{\frac{4}{d-2}}u =0 
\\
(u,   \pa_t u)\big|_{t = 0} = (u_0, u_1), 
\end{cases}
\quad \quad (t,x)\in\R\times\T^d, 
\end{equation} 

\noi
where 
$u$ is a real-valued function on $\R\times\T^d$.
In particular, we prove
almost sure global well-posedness of \eqref{NLW}
with randomized initial data 
 below the energy space.

NLW on the Euclidean space $\R^d$ has been studied extensively from both applied and theoretical points of view.
Due to its analytical difficulty, 
the energy-critical defocusing NLW \eqref{NLW} on $\R^d$ has attracted
a tremendous amount of attention 
over the last few decades. 
After substantial efforts  by many mathematicians, 
it is  known that
\eqref{NLW} on $\R^d$ is globally well-posed in the energy space
and all finite energy solutions scatter
 \cite{ Struwe, Grillakis90, Grillakis92, Shatah_Struwe93, 
Shatah_Struwe, Kapitanski, 
Ginibre, Bahouri_Shatah,
Bahouri_Gerard, 
Nakanishi1999,
Nakanishi_scattering,
Tao}.
Thanks to the finite speed of propagation,  
these global well-posedness results of \eqref{NLW} on $\R^d$ in the energy space
immediately yield 
the corresponding 
 global well-posedness of \eqref{NLW} on $\T^d$ in the energy space.
We point out that these well-posedness results in the energy space
are sharp in the sense that 
the energy-critical  NLW
\eqref{NLW} on $\R^d$ is known to be ill-posed
below the energy space
\cite{Christ_Colliander_Tao_main}.

In recent years, there has been a significant development
in incorporating non-deterministic points of view in 
the study of the Cauchy problems for 
hyperbolic and dispersive PDEs 
below certain regularity thresholds, in particular  a scaling critical regularity.
For example,  the methodology developed in \cite{Bourgain96, BTI, BOP1, Poc}
readily yields almost sure local well-posedness of \eqref{NLW} 
with respect to randomized initial data below the energy space. 
There are also results on almost sure global well-posedness
that go beyond the deterministic thresholds.
Burq-Tzvetkov \cite{BT3} considered the energy-subcritical 
defocusing cubic NLW on $\T^3$ and established 
almost sure global well-posedness below the scaling critical regularity.
Subsequently, L\"uhrmann-Mendelson \cite{LM} 
applied the probabilistic high-low method developed in \cite{Colliand_Oh}
and proved almost sure global well-posedness for some energy-subcritical NLW on $\R^3$
below the scaling critical regularity.
See \cite{LM2} for a recent improvement on this work.\footnote{There is also a recent work 
by Sun-Xia \cite{SX}
on almost sure global well-posedness for some energy-subcritical NLW on $\T^3$.}
More recently, 
the authors  \cite{Poc, OP}
incorporated the deterministic energy-critical theory
and proved almost sure global well-posedness below the energy space
of the energy-critical defocusing NLW \eqref{NLW} on $\R^d$, $d = 3, 4,$ and $ 5$.
Our main goal in this paper is to 
consider
the energy-critical defocusing NLW
\eqref{NLW} on $\T^d$ in the probabilistic setting 
and prove almost sure global well-posedness 
below the energy space.
In the classical deterministic setting, 
the finite speed of propagation immediately allows us to transfer
a deterministic global well-posedness result of NLW on $\R^d$
to the 
corresponding  deterministic global well-posedness result  on $\T^d$.
This finite speed of propagation also plays
an important role in our probabilistic setting.
As we see below, however, 
the probabilistic results on $\R^d$ in \cite{Poc, OP} 
are not directly transferrable to the periodic setting
and some care must be taken.

\subsection{Main result}

The energy-critical NLW  \eqref{NLW} on $\R^d$ is known to enjoy the following dilation symmetry:
$u(t,x)\mapsto u_{\lambda}(t,x):=\lambda^{\frac{d-2}{2}} u(\lambda t, \lambda x)$.
Namely, if $u$ is a solution to \eqref{NLW} on $\R^d$,
then $u_{\lambda}$ is also a solution to \eqref{NLW} on  $\R^d$ with rescaled initial data.
It is easy to check that 
the $\dot{H}^1(\R^d)\times L^2(\R^d)$-norm
and 
the   conserved energy $E(u)$ defined by 
\begin{equation*}
E(u) = E(u, \pa_t u) := \int \frac 12(\pa_t u)^2+\frac 12|\nabla u|^2+
\frac{d-2}{2d}
|u|^\frac{2d}{d-2} dx
\end{equation*} 

\noi
are invariant under 
this dilation symmetry.
Note that by Sobolev's inequality, 
$E(u, \dt u) < \infty$
if and only if $(u, \dt u)\in \dot{H}^1(\R^d)\times L^2(\R^d)$.
For this reason, 
the space $\dot{H}^1(\R^d)\times L^2(\R^d)$  is called the energy space.
While there is no dilation symmetry on $\T^d$, 
we still refer to $H^1(\T^d)\times L^2(\T^d)$
as the energy space for \eqref{NLW} posed on $\T^d$.

Our main goal is to prove almost sure global well-posedness
of \eqref{NLW} on $\T^d$ below the energy space.
We use the following shorthand notation
for products of  Sobolev spaces:
\[\mathcal{H}^s(M) : = H^s(M)\times H^{s-1}(M), 
\]

\noi
where $M = \T^d$ or $\R^d$.

Given $s < 1$, fix a pair $(u_0, u_1) \in \H^s(\T^d)$ of  real valued functions. 
In terms of the Fourier series, we have
\[ u_j(x) = \sum_{n \in \Z^d} \ft{u}_j(n) e^{in \cdot x}, \quad j = 0, 1, \]

\noi
such that $\ft u_j(-n) = \cj{\ft u_j(n)}$.
We introduce a randomization $(u_0^\o, u_1^\o)$ of $(u_0, u_1)$ as follows.
For $j = 0, 1$, 
let $\{g_{n,j}\}_{n \in \Z^d}$ be a sequence of mean zero complex-valued random variables
on a probability space $(\Omega, \mathcal{F}, P)$
such that $g_{-n,j}=\cj{g_{n,j}}$
for all $n\in\Z^d$, $j=0,1$.
In particular,  $g_{0, j}$ is real-valued.
Moreover, we  assume that 
$\{g_{0,j}, \Re g_{n,j}, \Im g_{n,j}\}_{n\in\mathcal I, j=0,1}$ are independent,
where the index set $\mathcal{I}$ is defined by 
\begin{equation}
\mathcal I:=\bigcup_{k=0}^{d-1} \Z^k\times \Z_{+}\times \{0\}^{d-k-1}.
\label{Index}
\end{equation}

\noi
Note that $\Z^d = \mathcal I \cup (-\mathcal I)\cup \{0\}$.
Then, we define the randomization $(u_0^\o, u_1^\o)$
of $(u_0, u_1)$ by 
\begin{align}
(u_0^\o, u_1^\o) :=
\bigg(\sum_{n \in \Z^d} g_{n, 0} \ft{u}_0(n) e^{in \cdot x},
 \sum_{n \in \Z^d} g_{n, 1} \ft{u}_1(n) e^{in \cdot x}\bigg).
\label{R1}
\end{align}

\noi
In particular, if
$\{g_{0,j}, \Re g_{n,j}, \Im g_{n,j}\}_{n\in\mathcal I, j=0,1}$ are independent
standard complex-valued Gaussian random variables, 
then the randomization \eqref{R1}  corresponds
the white noise randomization:
$(u_0^\o, u_1^\o) = (\Xi_0* u_0, \Xi_1*u_1)$, 
where $\Xi_0$ and $\Xi_1$ are independent Gaussian white noise on $\T^d$.
See \cite{OP} for more on this.

In the following, we also make the following assumption on the 
 probability distributions
$\mu_{n,j}$ of $g_{n, j}$;
 there exists $c>0$ such that
\begin{equation}
\int e^{\gamma \cdot x}d\mu_{n,j}(x)\leq e^{c|\gamma|^2}, \quad j = 0, 1, 
\label{cond}
\end{equation}
	
\noindent
for all $n \in \Z^d$,  
(i) all $\gamma \in \R$ when $n = 0$,
and (ii)  all $\g \in \R^2$ when $n \in \Z^d \setminus \{0\}$.
Note that \eqref{cond} is satisfied by
standard complex-valued Gaussian random variables,
standard Bernoulli random variables,
and any random variables with compactly supported distributions.

Our main result reads as follows.

\begin{theorem}\label{THM:GWP}
For  $d = 3, 4,$ or 5,
let 
 $s \in \R$ satisfy
\begin{align*}
\textup{(i)} \  
\tfrac{1}{2} < s < 1
\text{ when } d = 3, 
\quad 
\textup{(ii)} \ 0< s< 1
\text{ when } d = 4, 
\quad 
\textup{(iii)} \ 0\leq s< 1
\text{ when } d = 5. 
\end{align*}

\noi
Given   $(u_0, u_1) \in \mathcal{H}^s(\T^d)$, 
let $(u_0^\omega, u_1^\omega)$
be the  randomization defined in \eqref{R1},
satisfying \eqref{cond}. 
Then, the energy-critical defocusing  NLW \eqref{NLW}
on $\T^d$ is almost surely globally well-posed.
More precisely,
there exists a set $ \O_{(u_0, u_1)}\subset \Omega $ of probability 1
such that, 
for  every $\o \in \O_{(u_0, u_1)}$, there exists a unique solution $u^\o$ to \eqref{NLW}
with $(u^\omega, \dt u^\omega)|_{t = 0} = (u_0^\omega, u_1^\omega)$
in the class:
\begin{align*}
\big(S_\textup{per}(t)(u_0^\omega, u_1^\omega), \dt S_\textup{per}(t)(u_0^\omega, u_1^\omega)\big)
+C(\R; \mathcal{H}^1(\T^d))
\subset C(\R; \mathcal{H}^s(\T^d)).
\end{align*}

\end{theorem}

\noi
Here, $S_\textup{per}(t)$ denotes the propagator for the linear wave equation on $\T^d$ given by 
\begin{equation*}
S_\textup{per}(t)\left(f_0, f_1\right):=\cos(t|\nabla|)f_0+\frac{\sin (t|\nabla|)}{|\nabla|}f_1.
\end{equation*}

This is the first result on almost sure global existence of 
unique solutions to energy-critical  hyperbolic/dispersive PDEs
in the periodic setting.
In particular, when $ d= 4$, 
Theorem \ref{THM:GWP} provides an affirmative answer
to a question posed in \cite{BTT2}.
When $ d= 4$,
 Burq-Thomann-Tzvetkov \cite{BTT2} 
previously proved almost sure global existence
(without uniqueness)
of weak solutions to \eqref{NLW} on $\T^4$ 
for $0< s< 1$.
 Moreover, the continuity 
(of the nonlinear part) of the solution constructed in \cite{BTT2} was obtained only in a weaker topology. 
Their main approach was
to establish a probabilistic energy estimate
and apply a compactness argument.
The lack of uniqueness in \cite{BTT2} comes from the use of 
the compactness argument.
Theorem \ref{THM:GWP}
allows us to upgrade the weak solutions in \cite{BTT2}
to strong solutions.\footnote{Here, we are indeed referring to the nonlinear part
 of a solution $u$.}

In the Euclidean setting,
 we introduced in \cite{Poc, OP} the probabilistic perturbation theory
and proved  almost sure global existence 
 of unique solutions to \eqref{NLW}
on $\R^d$, $d = 3, 4,$ and $ 5$.
Let us briefly discuss the randomization of real-valued functions on $\R^d$
employed in \cite{Poc, OP}.
Let $\psi \in \mathcal{S}(\R^d)$ 
be such that $\supp \psi \subset [-1, 1]^d$,  $\psi(-\xi ) = \cj{\psi(\xi)}$,
and 
\begin{align}
 \sum_{n \in \Z^d} \psi(\xi - n) \equiv 1 \quad \text{for all }\xi \in \R^d.
\label{Zpsi}
 \end{align}

\noi
Then, any function $u$ on $\R^d$ can be written as
\begin{equation}
 u = \sum_{n \in \Z^d} \psi(D-n) u,
\label{Ziv2}
 \end{equation}

\noi
where $ \psi (D-n) $ denotes the Fourier multiplier operator
with symbol $\psi (\,\cdot\, -n)$.
We then consider  a randomization 
adapted to the decomposition \eqref{Ziv2}.
More precisely, given a pair $(u_0, u_1)$ of  functions on $\R^d$, 
we  define the Wiener randomization $(u_0^\omega, u_1^\omega)$
of $(u_0,u_1)$ by
\begin{align}
(u_0^\omega, u_1^\omega)  := 
\bigg(\sum_{n \in \Z^d} g_{n,0} (\omega) \psi(D-n) u_0,
\sum_{n \in \Z^d} g_{n,1} (\omega) \psi(D-n) u_1\bigg).
\label{RR1}
\end{align}

\noi
This randomization is based on the uniform decomposition
of  the frequency space $\R^d_\xi$ 
into the unit cubes, called the Wiener decomposition \cite{W}.
In \cite{Poc, OP}, we proved that, given 
 $s < 1$ satisfying the condition in Theorem \ref{THM:GWP}
 and any $(u_0, u_1) \in \H^s(\R^d)$, 
the energy-critical defocusing NLW on $\R^d$
is  almost surely globally well-posed with respect 
to the Wiener randomization $(u_0^\o, u_1^\o)$ defined in \eqref{RR1}.
See \cite{LM, BOP1, BOP2, LM2}
for other results utilizing the Wiener randomization \eqref{RR1}.

Our basic strategy for the proof of Theorem \ref{THM:GWP}
is to make use of the finite speed of propagation of solutions
and reduce the problem on $\T^d \cong \big[-\frac 12, \frac 12\big)^d$ to a problem in the Euclidean setting.
Fix $\eta \in C^\infty_c(\R^d; \R)$
such that $\eta \equiv 1$ on $[-1, 1]^d $.
Given $T > 0$, let 
\begin{align}
\eta_{_T}(x) = \eta\big(\jb{T}^{-1}x\big),
\label{eta}
\end{align}

\noi
\noi
where $\jb{\, \cdot\,} = 1 + |\cdot|$.
Let $\bu$ be a solution to the following energy-critical defocusing NLW
on $\R^d$:
\begin{equation}
\label{NLW2}
\begin{cases}
\pa_{t}^2 \bu -\Delta \bu+|\bu|^{\frac{4}{d-2}}\bu =0 
\\
(\bu,   \pa_t \bu)\big|_{t = 0} = (\bu_{0, T}, \bu_{1, T}):=(\et u_0, \et u_1), 
\end{cases}
\quad \quad (t,x)\in [0, T] \times \R^d, 
\end{equation} 

\noi
where we view $(u_0, u_1)$ as periodic functions on $\R^d$ with period 1.
Then, by the finite speed of propagation, 
we see that $u: =  \bu |_{[0, T]\times \T^d}$
is a solution to the periodic NLW \eqref{NLW} on the time interval $[0, T]$
with initial data $(u_0, u_1)$.
In the classical deterministic setting, 
this allows us to transfer global well-posedness on  NLW on $\R^d$
to the corresponding 
global well-posedness of the periodic NLW  on $\T^d$.
In our current probabilistic setting, 
however, 
this is not so straightforward. 
In particular, under such a reduction
from the periodic setting to the Euclidean setting, 
our random initial data $(u_0^\o, u_1^\o)$ on $\T^d$ of the form  \eqref{R1} 
does not give rise to 
an appropriate random initial data on $\R^d$ of the form \eqref{RR1} such that 
the results in \cite{Poc, OP} are directly applicable.

Fix  a pair   $(u_0, u_1)$ of real-valued functions defined on $\H^s(\T^d)$.
Given $T > 0$, 
define 
a pair $(\bu_{0, T}^\o, \bu_{1, T}^\o)$ of random functions on $\R^d$ 
by setting
\begin{align}
(\bu_{0, T}^\o, \bu_{1, T}^\o)
:& \!=(\et u_0^\o, \et u_1^\o)\notag \\
& = \bigg( \sum_{n \in \Z^d} \et (x) g_{n, 0}(\o) \ft u_0(n) e^{in\cdot  x},
\sum_{n \in \Z^d} \et (x) g_{n, 1}(\o) \ft u_1(n) e^{in\cdot x}\bigg),
\label{R3}
\end{align}

\noi
where $\et$ is as in \eqref{eta}
and 
$(u_0^{\omega}, u_1^\omega)$ is the randomization of $(u_0, u_1)$ defined in \eqref{R1}, satisfying \eqref{cond}. 
Then, in order to prove Theorem \ref{THM:GWP}, we need  to prove almost sure
well-posedness of \eqref{NLW2} on $[0, T] \times \R^d $ with 
$(\bu,   \pa_t \bu)\big|_{t = 0} = (\bu_{0, T}^\o, \bu_{1, T}^\o)$
for some sequence of  $T \to \infty$.
First, note that the randomized initial data 
$(\bu_{0, T}^\o, \bu_{1, T}^\o)$ in \eqref{R3} depends on $T$.
Moreover, it is not of the form \eqref{RR1}.
Indeed, we have 
\begin{align}
\wt{\bu_{j, T}^\o}(\xi) = \wt{\et u_j^\o}(\xi)  = \sum_{n \in \Z^d}  \fet(\xi - n) g_{n, j}(\o) \ft u_j(n),
\quad j = 0, 1.
\label{Hs0}
\end{align}

\noi
In particular, the Fourier transform $\wt{\bu_{j, T}^\o}(\xi) $
depends on infinitely many $g_{n, j}$'s for each $\xi \in \R^d$.  See Remark \ref{REM:PW} below.

The proof of almost sure global well-posedness of \eqref{NLW} on $\R^d$  in \cite{Poc, OP}
consists of two disjoint parts:
 (i) a probabilistic part
and (ii) a deterministic part.
We can apply the deterministic part of the argument without any change.
Therefore, our main task is to 
adapt   the probabilistic part to our current problem.
In particular, we will establish probabilistic Strichartz estimates
(Propositions \ref{PROP:Str} and  \ref{PROP:infty} below)
that allow us to control random linear profiles on $\R^d$
in terms of functions on $\T^d$.
See Section \ref{SEC:Str}.
We then need to adjust the argument in \cite{Poc, OP} suitably
to our setting.

We conclude this introduction by stating several remarks.

\begin{remark}\label{REM:PW} \rm
If there were a function $\eta \in L^2(\R^d)$ with the properties
(i) $\eta(x) \equiv 1$ on $[-\frac 12, \frac 12)^d$
and (ii) its Fourier transform $\wt \eta$ has a compact support, 
then we could basically apply the arguments in \cite{Poc, OP}
to study \eqref{NLW2} with random initial data $(\bu_{0, T}^\o, \bu_{1, T}^\o)$ defined in \eqref{R3}.
However, 
Paley-Wiener Theorem
(Theorems IX.11 and IX.12 in \cite{RS})
states that there is no such function $\eta \in L^2(\R^d)$
satisfying both (i) and (ii).

\end{remark}

\begin{remark}\label{REM:uniq} \rm
 The  uniqueness statement  in Theorem \ref{THM:GWP} holds
in the following sense.
The existence part of Theorem \ref{THM:GWP}
states that 
given any $\o\in \O_{(u_0, u_1)} $, 
 there exists a global solution $u^\o$ to \eqref{NLW}.
Now, we fix one such $\o \in \O_{(u_0,u_1)}$
and let $f^\o := S_\textup{per}(\cdot)(u_0^\omega, u_1^\omega)$.
%
%
Setting 
 $v^\o := u^\o - f^\o$, we see that $v^\o$ is a global solution to 
the perturbed NLW  on $\T^d$:
\begin{align}
\begin{cases}
\pa_{t}^2v^\o-\Delta v^\o+|v^\o+f^\o|^{\frac{4}{d-2}}(v^\o+f^\o)=0\\
(v^\o, \dt v^\o)|_{t = 0} = (0, 0).
\end{cases}
\label{uniq2}
\end{align}

\noi
Then, the  uniqueness in Theorem \ref{THM:GWP} holds for $v^\o$ 
in 
\begin{align}
 X(\R) := \big\{ (v, \dt v): (v, \dt v) \in C(\R,\dot \H^1(\T^d)),
\ v \in L_{\rm loc}^\frac{d+2}{d-2}(\R, L^{\frac{2(d+2)}{d-2}}(\T^d))\big\}.
\label{XX}
\end{align}

\noi
This follows from 
a standard deterministic analysis of the perturbed NLW \eqref{uniq2} on $\T^d$.
See Appendix \ref{SEC:uniq}.
In terms of $u^\o$, the uniqueness
holds 
in 
\[ \big( S_\textup{per}(t)(u_0^\omega, u_1^\omega), 
 \dt S_\textup{per}(t)(u_0^\omega, u_1^\omega)\big) + X(\R).\]

\noi
Lastly, note that the almost sure global solutions constructed in \cite{Poc, OP}
also satisfy the same kind of  uniqueness.
\end{remark}

\begin{remark}\rm
Let  ${\bf u_0}:\O\to
\mathcal{H}^s(\T^d)$
be the map 
given by 
${\bf u_0}(\omega) : =  (u_0^\omega,u_1^\omega)$,
where $(u_0^\omega,u_1^\omega)$ 
is as  in \eqref{R1}.
Then, the map  ${\bf u_0}$
induces a probability measure $\mu = \mu_{(u_0, u_1)} = P\circ {\bf u_0}^{-1}$
on $\mathcal{H}^s(\T^d)$.
Now, let  $\Si_{(u_0, u_1)} = {\bf u_0}(\O_{(u_0, u_1)})$, 
where $\O_{(u_0, u_1)}$ is as in Theorem \ref{THM:GWP}.
Then, while $\mu(\Sigma_{(u_0,u_1)})=1$, 
it is possible that 
 $\mu\big(\Phi(t)(\Sigma_{(u_0, u_1)})\big) $ becomes smaller for some $t \ne 0$ and even tends to  0, 
 where
 $\Phi(t)$ denotes the solution map of \eqref{NLW}.
Arguing as in \cite{Poc}, 
we can strengthen 
 the statement in Theorem \ref{THM:GWP}
and
show that there exists another set of $\mu$-full measure $\Sigma \subset \mathcal{H}^s(\T^d)$
such that 
(a)  for any $(\phi_0,\phi_1)\in \Sigma$,
there exists a unique global solution $u$ to  \eqref{NLW}  with initial data 
$\left(u,\pa_t u\right)\big|_{t=0}=(\phi_0, \phi_1)$
and (b) 
 $\mu\big(\Phi(t)(\Sigma)\big)  = 1 $ for any $t \in \R$.
 Namely, the measure of our new initial data set $\Si$ does not become smaller
under the dynamics of \eqref{NLW}. 
See  \cite{BT3, OQ, Poc} for related discussions in this direction.

\end{remark}

\section{Notations}
Given a periodic function $f$ on $\T^d$, 
we use $\ft{f}(n) = \F_{\T^d}(f)(n)$ to denote the Fourier coefficient of $f$ on $\T^d$.
Given a function $f$ on $\R^d$, 
we use $\wt{f}(\xi) = \F_{\R^d}(f)(\xi)$ to denote the Fourier transform of $f$ on $\R^d$.
Let $f$ be a periodic function on $\T^d$.
By viewing $f$ as a tempered distribution on $\R^d$
we have 
\[ \wt f (\xi) = \sum_{n \in \Z^d} \dl (\xi - n) \ft f(n) .\]

\noi
Moreover, given $\eta \in \S(\R^d)$, we have 
\begin{align}
 \wt {\eta f} (\xi) = \sum_{n \in \Z^d} \wt \eta (\xi - n) \ft f(n).
\label{FT1}
 \end{align}

\noi
Given $n \in \Z^d$, 
let $Q_n$ be  the unit cube $Q_n:=n+\big[-\frac 12, \frac 12 \big)^d$ centered at $n$.

Next, we briefly go over 
the Littlewood-Paley theory on $\R^d$.
Let $\varphi:\R \to [0, 1]$ be a smooth  bump function supported on $[-\frac{8}{5}, \frac{8}{5}]$ 
and $\varphi\equiv 1$ on $\big[-\frac 54, \frac 54\big]$.
Given dyadic $N \geq1$, 
we set $\varphi_1(\xi) = \varphi(|\xi|)$
and 
\[\varphi_N(\xi) = \varphi\big(\tfrac{|\xi|}N\big)-\varphi\big(\tfrac{2|\xi|}N\big)\]

\noi
for $N \geq 2$.
Then, we define the Littlewood-Paley projection $\P_N$
as the Fourier multiplier operator with symbol $ \varphi_N$.
Moreover, we define $\P_{\leq N}$ and $\P_{\geq N}$
by 
 $\P_{\leq N} = \sum_{1\leq M \leq N} \P_M$ and $\P_{> N} = \sum_{M> N} \P_M$.
For a periodic function $f$ on $\T^d$, we
define $\P_N$ 
to be the projection onto the frequencies $ \{\frac 12 N < |n| \leq N\}$ if $N \geq 2$
and $ \{|n| \leq 1\}$ if $N =1$.
In the following, we use $\P_N$ to denote
the Littlewood-Paley projection
for both functions on $\R^d$ and $\T^d$, depending on the context.

We use $S(t)$ to denote the propagator for the linear wave equation on $\R^d$ 
given by 
\begin{equation}\label{Zlinear}
S(t)\left(f_0, f_1\right):=\cos(t|\nabla|)f_0+\frac{\sin (t|\nabla|)}{|\nabla|}f_1.
\end{equation}

\noi
We say that $u$ is a solution to the following nonhomogeneous wave equation on $\R^d$:
\begin{equation}
\begin{cases}
\pa_{t}^2u-\Delta u+F=0\\
(u, \dt u)|_{t = t_0} = (\phi_0, \phi_1)
\end{cases}
\label{Wave}
\end{equation}

\noi
on a time interval $I$ containing $t_0$, 
if $u$ satisfies  the following Duhamel formulation:
\begin{equation}
u(t)=S(t-t_0)(\phi_0, \phi_1)
-\int_{t_0}^t \frac{\sin ((t-t')|\nabla|)}{|\nabla|}F(t')dt'
\label{Duhamel}
\end{equation}

\noi
for $t \in I$.
We now recall the Strichartz estimates for wave equations on $\R^3$.
 We say that $(q,r)$ is a $s$-wave admissible pair
if $q\geq 2$, $2\leq r<\infty$,
\[\frac{1}{q}+\frac{d-1}{2r}\leq \frac{d-1}{4},
\quad \text{and}\quad 
\frac 1q+\frac dr=\frac d2-s.\]

\noi
Then, we have the following Strichartz estimates.
See \cite{Ginibre, Lindblad, Keel}
for more discussions on the  Strichartz estimates. 

\begin{lemma}
\label{LEM:Strichartz}
Let $s>0$. 
Let  $(q,r)$ and $(\tilde{q},\tilde{r})$ be $s$- and $(1-s)$-wave admissible pairs,
respectively. 
Then, we have
\begin{align}\label{Strichartz}
\|(u, \dt u)\|_{L^\infty_t(I;\dot{\mathcal{H}}^s_x(\R^d))}+\|u\|_{L^q_t(I; L^r_x)}\lesssim
\|(\phi_0, \phi_1) \|_{\dot{\mathcal{H}}^{s}(\R^d)}+\|F\|_{L^{\tilde{q}'}_t(I; L^{\tilde{r}'}_x(\R^d))}
\end{align}

\noi
for all solutions  $u$ to \eqref{Wave} on a time interval $I \ni t_0$.

\end{lemma}

\noi
In our argument, 
we will only use 
the following
wave admissible pairs: $\big(\frac{d+2}{d-2}, \frac{2(d+2)}{d-2}\big)$ with $s = 1$
and 
 $(\infty,2)$ with $s = 0$.
For simplicity,  we  denote the space $L^q_t(I; L^r_x)$  by $L^q_IL^r_x$
or  $L^q_TL^r_x$ if $I=[0,T]$.

In the following, constants in various estimates depend on 
the smooth cutoff function $\eta$, appearing in \eqref{eta}.
Since we fix such $\eta$ once and for all, 
we suppress the dependence on $\eta$.
Lastly, in view of  the time reversibility of the equation, 
we only consider positive times in the following.

\section{Reduction to the Euclidean setting}
\label{SEC:reduction}

We first reduce
Theorem \ref{THM:GWP}
to 
the following proposition 
on ``almost'' almost sure global well-posedness of \eqref{NLW}.

\begin{proposition}\label{PROP:GWP}
Let $(s, d)$ be as in Theorem \ref{THM:GWP}.
Given   $(u_0, u_1) \in \mathcal{H}^s(\T^d)$, 
let $(u_0^\omega, u_1^\omega)$
be the  randomization defined in \eqref{R1},
satisfying \eqref{cond}. 
Then, for any given $T\geq 1$ and $\eps > 0$, 
there exists a set $ \O_{T, \eps}\subset \Omega $ with $P(\O_{T, \eps}^c) < \eps$ 
such that, 
for  every $\o \in \O_{T, \eps}$, there exists a unique solution $u^\o$ to \eqref{NLW}
with 
$(u^\omega, \dt u^\omega)|_{t = 0} = (u_0^\omega, u_1^\omega)$
in the class:
\begin{align}
\big(S_\textup{per}(t)(u_0^\omega, u_1^\omega), \dt S_\textup{per}(t)(u_0^\omega, u_1^\omega)\big)
+C([0, T]; \mathcal{H}^1(\T^d))
\subset C([0, T]; \mathcal{H}^s(\T^d)).
\label{class1}
\end{align}

\end{proposition}

\noi
It is easy to see that Proposition \ref{PROP:GWP} implies Theorem \ref{THM:GWP}.
See, for example, \cite{Colliand_Oh, Poc}.
Therefore, in the remaining part of the paper, 
we focus on the proof of Proposition \ref{PROP:GWP}
for each fixed $T\geq 1$ and $\eps > 0$.

Given $(u_0, u_1)\in \H^s(\T^d)$
and $T\geq 1$, 
let  $(\bu_{0, T}^\o, \bu_{1, T}^\o)$ be the random functions on $\R^d$ defined  in \eqref{R3}.
Consider the following Cauchy problem:
\begin{equation}
\label{NLW3}
\begin{cases}
\pa_{t}^2 \bu^\o -\Delta \bu^\o+|\bu^\o|^{\frac{4}{d-2}}\bu^\o =0 
\\
(\bu^\o,   \pa_t \bu^\o)\big|_{t = 0} = (\bu_{0, T}^\o, \bu_{1, T}^\o), 
\end{cases}
\quad \quad (t,x)\in [0, T] \times \R^d.
\end{equation} 

\noi
In view of  the finite speed of propagation, 
Proposition \ref{PROP:GWP} follows once
we prove the following proposition.
See Appendix \ref{SEC:APP} for this part of the reduction.

\begin{proposition}\label{PROP:GWP2}
Let $(s, d)$ be as in Theorem \ref{THM:GWP}.
Given   $(u_0, u_1) \in \mathcal{H}^s(\T^d)$ and $T\geq 1$, 
let $(\bu_{0, T}^\o, \bu_{1, T}^\o)$ be the random functions on $\R^d$ defined in \eqref{R3}, 
satisfying \eqref{cond}. 
Then, for any $\eps > 0$, 
there exists a set $ \wt \O_{T, \eps}\subset \Omega $ with $P(\wt \O_{T, \eps}^c) < \eps$ 
such that, 
for  every $\o \in \wt \O_{T, \eps}$, there exists a unique solution $\bu^\o$ to \eqref{NLW3}
with $(\bu^\omega, \dt \bu^\omega)|_{t = 0} = (\bu_{0, T}^\omega, \bu_{1, T}^\omega)$
in the class:
\begin{align*}
\big(S(t)(\bu_{0, T}^\omega, \bu_{1, T}^\omega), \dt S(t)(\bu_{0, T}^\omega, \bu_{1, T}^\omega)\big)
+C([0, T]; \mathcal{H}^1(\R^d))
\subset C([0, T]; \mathcal{H}^s(\R^d)).
\end{align*}
Moreover, the nonlinear part 
$\bv^\omega:=\bu^\omega-S(\cdot)(\bu_{0, T}^{\omega}, \bu_{1, T}^{\omega})$
of the solution
satisfies the bounds
\begin{align}\label{Stricbv}
\|\bv^\omega\|_{L^q_t([0,T], L^r_x(\R^d))}\leq C(T,\eps,\|(u_0,u_1)\|_{\mathcal{H}^s(\T^d)}),
\end{align}
for all $1$-wave admissible pairs $(q,r)$.

\end{proposition}

The main idea is to adapt the argument in 
\cite{Poc, OP}
on almost sure global well-posedness of \eqref{NLW} on $\R^d$
with random initial data of the form \eqref{RR1}.
Denoting the linear
 and nonlinear parts of the solution $\bu^\o$
to  \eqref{NLW3}
by 
\begin{equation}
\bz^{\omega}(t) =
\bz_T^{\omega}(t):=
S(t)(\bu_{0, T}^{\omega}, \bu_{1, T}^{\omega})
\qquad \text{and}
\qquad \bv^\omega:=\bu^\omega-\bz^\omega, 
\label{z}
\end{equation}

\noi
we can reformulate  \eqref{NLW3} as the following  perturbed
NLW: 
\begin{equation}\label{NLW4}
\begin{cases}
\pa_{t}^2 \bv^\omega-\Delta \bv^\omega+F(\bv^\omega+\bz^\omega)=0\\
(\bv^\omega,\pa_t \bv^\omega)|_{t=0}=(0,0), 
\end{cases}
\end{equation}

\noi
where $F(u) = |u|^{\frac{4}{d-2}}u$.
As mentioned above,  the argument in \cite{Poc, OP}
can be divided into two parts:
 (i) the probabilistic part
and (ii) the deterministic study of the perturbed NLW: 
\begin{equation}\label{NLW5}
\begin{cases}
\pa_{t}^2 \bv-\Delta \bv+F(\bv+f) = 0
\\
(\bv,   \pa_t \bv)\big|_{t = 0} = (\bv_0, \bv_1), 
\end{cases}
\end{equation} 

\noi
where  $f$ is a deterministic function, 
satisfying some a priori space-time bounds. 
This deterministic part (Proposition 4.3 in \cite{Poc} and Proposition 5.2 in \cite{OP}) can be
applied to our problem without any change,
and hence we take it as a black box in this paper.
Therefore, our main task is to appropriately modify the probabilistic part of the argument.

In the next section, we prove new probabilistic Strichartz estimates
(Propositions \ref{PROP:Str} and  \ref{PROP:infty}), 
controlling the size of 
the random linear solution 
$S(t)(\bu_{0, T}^{\omega}, \bu_{1, T}^{\omega})$
on $\R^d$ in terms of the deterministic initial data $(u_0, u_1)$ on $\T^d$.
Then, in 
Section \ref{SEC:4d}, 
we briefly discuss how to modify the argument in \cite{Poc, OP}
to prove Proposition \ref{PROP:GWP2}.
In Appendix \ref{SEC:APP}, 
we consider the issue on the finite speed of propagation
on random solutions at a low regularity
and show how to deduce Proposition \ref{PROP:GWP} from Proposition \ref{PROP:GWP2}.
Finally, 
we sketch the uniqueness part of Theorem \ref{THM:GWP}
in Appendix \ref{SEC:uniq}.

Let us conclude this section by stating a lemma, 
which 
 allows us to compare the $H^s$-norms
 of a periodic function 
on $\R^d$ and $\T^d$ through the multiplication by $\et$.

\begin{lemma}\label{LEM:Hs}

Let $0 \leq s <1$.
Then, there exist $C>0$ such that 
\begin{align}
\frac{1}{C}  \jb{T}^\frac{d}{2} \| f \|_{H^s(\T^d)}\|
\leq \|\et f\|_{H^s(\R^d)} \leq C \jb{T}^\frac{d}{2} \| f \|_{H^s(\T^d)}, 
\label{embed0}
\end{align}
	
\noi
for any $T >0$ and  any periodic function $f\in H^s(\T^d)$.
\end{lemma}

\begin{proof}

Given $m \in \Z^d$ and $T >0 $, set 
 $\jb{T}Q_m := \{ \xi \in \R^d: \, \jb{T}^{-1} \xi \in  Q_m \}$.
Then, for  $s\geq 0$, it follows from \eqref{eta} that
\begin{align}
\label{embed1}
\|\jb{\,\cdot \,}^{s}\fet \|_{L^2(Q_m)}
\leq \jb{T}^\frac{d}{2} 
\|\jb{\,\cdot \,}^{s}\wt \eta \|_{L^2(\jb{T} Q_m)}
\les \jb{T}^\frac{d}{2} 
\sum_{\substack{k \in \Z^d\cap\jb{T}Q_m}} \jb{k}^s 
\|\psi(D-k) \eta \|_{L^2(\R^d)}, 
\end{align}

\noi
where $\psi$ is as in \eqref{Zpsi}.
Then, 
by \eqref{FT1}, the triangle inequality (with $s \geq 0$),  
Minkowski's integral inequality, Young's inequality, and \eqref{embed1}, we have 
\begin{align}
\|\et f\|_{H^s(\R^d)} 
& = \Bigg(\int \jb{\xi}^{2s} \bigg|
 \sum_{n \in \Z^d}  \fet(\xi - n)  \ft f(n)\bigg|^2d\xi\Bigg)^\frac{1}{2} \notag \\
& \les \Bigg(\int \bigg(
 \sum_{n \in \Z^d}  \jb{\xi-n}^{s} |\fet(\xi - n)| 
\cdot  \jb{n}^{s}  |\ft f(n)|\bigg)^2d\xi\Bigg)^\frac{1}{2} \notag \\
& \les \Bigg(\sum_{m \in \Z^d} \int_{Q_m} 
 \bigg(
 \sum_{n \in \Z^d}  \jb{\xi-n}^{s} |\fet(\xi - n)| 
\cdot  \jb{n}^{s}  |\ft f(n)|\bigg)^2 d\xi\Bigg)^\frac{1}{2}\notag \\
& \les \Bigg(\sum_{m \in \Z^d} 
 \bigg(
 \sum_{n \in \Z^d} 
 \|\jb{\,\cdot \,}^{s}\fet \|_{L^2(Q_{m-n})}
 \jb{n}^{s}  |\ft f(n)|\bigg)^2 \Bigg)^\frac{1}{2}\notag \\
& \leq 
 \|\jb{\,\cdot\, }^{s}\fet \|_{\l^1_m L^2(Q_m)} \|f\|_{H^s(\T^d)}
 \les \jb{T}^\frac{d}{2}\|\eta\|_{M^s_{2, 1}} \|f\|_{H^s(\T^d)}.
\label{embed1a}
\end{align}

\noi
Here, $M^s_{2, 1}$
denotes the (weighted) modulation space defined by the norm
\[ \|\eta\|_{M^s_{2, 1}} = 
\Big\| \jb{n}^{s} \| \psi(D-n) \eta\|_{L^2(\R^d)}\Big\|_{\l^1(\Z^d)}.\]

Let $\T_{_T}^d : = \big[-T-\frac 12 , T+\frac 12\big)^d$.
Then, by the definition of $\et$, we have 
\begin{align}
\jb{T}^\frac{d}{2} \| f \|_{L^2(\T^d)}
\sim \| f \|_{L^2(\T_{_T}^d)}
\leq \| \et f \|_{L^2(\R^d)}.
\label{embed2}
\end{align}

\noi
By the characterization of the $\dot H^s$-norms on the physical side
(see, for example,  \cite{Ho} and \cite{BO2} on $\R^d$ and  $\T^d$, respectively), the periodicity of $f$,
and the definition of $\et $, we have 
\begin{align}
\jb{T}^\frac{d}{2} \| f \|_{\dot H^s(\T^d)}
& \sim
\jb{T}^\frac{d}{2}
\bigg(\int_{\T^d} \int_{Q_0} \frac{|f(x+y) - f(x)|^2}{|y|^{d + 2s}}dy dx\bigg)^\frac{1}{2} \notag \\
& \sim
\bigg(\int_{\T_{T}^d} \int_{Q_0} \frac{|f(x+y) - f(x)|^2}{|y|^{d + 2s}}dy dx\bigg)^\frac{1}{2}\notag \\
& \leq 
\bigg(\int_{\R^d} \int_{Q_0} \frac{|\et(x+y) f(x+y) -\et(x) f(x)|^2}{|y|^{d + 2s}}dy dx\bigg)^\frac{1}{2}\notag\\
& \leq 
\bigg(\int_{\R^d} \int_{\R^d} \frac{|\et(x+y) f(x+y) -\et(x) f(x)|^2}{|y|^{d + 2s}}dy dx\bigg)^\frac{1}{2}\notag\\
& \sim \|\et f \|_{\dot H^s(\R^d)}
\label{embed3}
\end{align}

\noi
for $0 < s < 1$.	
Hence,  \eqref{embed0}
follows from \eqref{embed1a}, \eqref{embed2},  and \eqref{embed3}.
\end{proof}

\section{Probabilistic Strichartz estimates}
\label{SEC:Str}

In this section, we state and prove the crucial probabilistic Strichartz estimates
that build a bridge between the random linear solution
$S(t)(\bu_{0, T}^{\omega}, \bu_{1, T}^{\omega})$
on $\R^d$ and the  deterministic initial data $(u_0, u_1)$ on $\T^d$.
In  \cite{Poc, OP}, 
we studied the probabilistic Strichartz estimates
on $\R^d$ with random initial data of the form \eqref{RR1} (Proposition 2.3 in \cite{Poc}
and Proposition 3.3 in \cite{OP}).
The following propositions (Propositions \ref{PROP:Str} and  \ref{PROP:infty})
are suitable replacements  for our problem at hand.
In particular, we have the $\H^s(\T^d) $-norm of $(u_0, u_1)$ on $\T^d$
on the right-hand side of \eqref{Str0} and \eqref{infty00}.

\begin{proposition}\label{PROP:Str}
Let $T>0$.
Given  $(u_0, u_1)\in \H^s(\T^d)$,
let $(\bu_{0, T}^\o, \bu_{1, T}^\o)$ 
be the randomization on $\R^d$  defined in \eqref{R3}, satisfying \eqref{cond}. 
 Then, given $1\leq q< \infty$, $2\leq r\leq \infty$,
there exist $C,c>0$ such that
\begin{align}
P\Big(\|S(t)(\bu_{0, T}^\omega, \bu_{1, T}^\omega& )\|_{L^q_t(I; L^r_x(\R^d))}>\ld\Big) \notag \\
& \leq C\exp\Bigg(-c\frac{\ld^2}{\max(1,  b^2) \jb{T}^d |I|^{\frac2q}
\|(u_0, u_1)\|_{\mathcal{H}^s(\T^d)}^2}\Bigg)
\label{Str0}
\end{align}

\noi
for  any compact time interval $I=[a,b]\subset [0, T]$, 
provided 
\textup{(i)} $s = 0$ if $r < \infty$
and \textup{(ii)}  $s > 0$ if $r = \infty$.

\end{proposition}

\begin{remark}\rm
Let $ (\bu_{0, T}, \bu_{1, T}):=(\et u_0, \et u_1)$ as in \eqref{NLW2}.
Then, in view of Lemma \ref{LEM:Hs}, we can rewrite \eqref{Str0} as 
\begin{align}
P\Big(\|S(t)(\bu_{0, T}^\omega, \bu_{1, T}^\omega& )\|_{L^q_t(I; L^r_x(\R^d))}>\ld\Big) \notag \\
& \leq C\exp\Bigg(-c\frac{\ld^2}{\max(1,  b^2)  |I|^{\frac2q}
\|(\bu_{0, T}, \bu_{1, T})\|_{\mathcal{H}^s(\R^d)}^2}\Bigg).
\label{Str0a}
\end{align}

\noi
We point out that \eqref{Str0a} 
is more in the spirit of  the statement of Proposition 2.3 (ii) and (iii) in \cite{Poc}.

\end{remark}

Before presenting the proof of Proposition \ref{PROP:Str}, 
we
first recall the following probabilistic estimate.
See \cite{BTI} for the proof.

\begin{lemma}\label{LEM:HC}
Let $\{g_n\}_{n\in \Z^d}$ be a sequence of mean zero complex-valued
 random variables  
such that $g_{-n}=\overline{g_n}$ for all 
$n\in\Z^d$.
With $\mathcal{I}$  as in \eqref{Index}, 
assume that $g_0$, $\Re g_n$, and $\Im  g_n$,
$n\in\mathcal I$, 
are independent. 
Moreover, 
assume that 
\eqref{cond} is satisfied.
Then, there exists $C>0$ such that the following holds:
\begin{equation*}
\Big\|\sum_{n\in\Z^d}g_n(\omega)c_n\Big\|_{L^p(\Omega)}\leq C\sqrt{p} \|c_n\|_{\ell^2_n(\Z^d)}
\end{equation*}

\noi
for any $p\geq 2$
and any sequence $\{c_n\} \in \ell^2 (\mathbb{Z}^d)$ satisfying $c_{-n}=\overline{c_n}$
for all $n\in\Z^d$.
\end{lemma}

\begin{proof}[Proof of Proposition \ref{PROP:Str}]
The proof is analogous to that of Proposition 2.3 in \cite{Poc}.
There are, however, important differences 
due to the fact that $\wt \eta$ does not have a compact support
and that we use the $\H^s(\T^d)$-norm of $(u_0, u_1)$
on the right-hand side of \eqref{Str0}.

\smallskip

\noi
$\bullet$ {\bf Case 1:} We first consider the case $r < \infty$.
Given  $1\leq q<\infty$ and  $2\leq r <\infty$, 
let  $p\geq\max (q,r)$.

Let $T_m$ be the Fourier multiplier operator
with a bounded multiplier $m$.
Let $\beta = \frac{(r-2)}{2r} d+\eps \leq \frac d2 $
for some small $\eps > 0$.
Then, 
by Hausdorff-Young's inequality
and H\"older's inequality
with $\frac{1}{r'} = \frac{1}{2} + \frac{r-2}{2r}$, 
we have
\begin{align}
\| T_m (\et e^{inx}) \|_{L^r_x(\R^d)}
& \leq 
\| m(\xi)  \fet(\xi-n)\|_{L^{r'}_\xi(\R^d)}
 \les 
\| \jb{\xi - n}^\be m(\xi)  \fet(\xi-n)\|_{L^{2}_\xi(\R^d)}\notag \\
& \les \|m\|_{L^\infty(\R^d)} \| \et  \|_{H^{\frac{d}{2}}(\R^d)}
\les \jb{T}^\frac{d}{2} \|m\|_{L^\infty(\R^d)}\| \eta  \|_{H^{\frac{d}{2}}(\R^d)}
\label{Str2}
\end{align}

\noi
for each $n \in \Z^d$.
Note that we have  
\begin{align}
|\cos (t|\xi|)|\leq 1\qquad 
\text{and}
\qquad 
\bigg|\frac{\sin (t|\xi|)}{|\xi|}\bigg|\leq t
\label{Str3}
\end{align}

\noi
 for all $\xi \in \R^d \setminus \{0\}$
and all $t \in \R$.
Then,  
by Minkowski's integral inequality, Lemma \ref{LEM:HC}
with  \eqref{R3},  and \eqref{Str2} with \eqref{Str3}, we have 
\begin{align}
\Big(\E  \big\| & \cos(  t|\nb|)
 \bu_{0, T}^\omega
\big\|^p_{L^q_t(I; L^r_x(\R^d))}\Big)^{\frac 1p}
\leq \Big\|\|\cos(t|\nb|)
\bu_{0, T}^\omega\|_{L^p(\Omega)}\Big\|_{L^q_IL^r_x} \notag \\
&\les \sqrt{p}  \Big\|\big\|\ft u_0(n) \cos (t|\nabla|)(\et   e^{inx})\big\|_{\ell^2_n}\Big\|_{L^q_IL^r_x}
\leq \sqrt{p}  \Big\|\big\| \cos (t|\nabla|)(\et   e^{inx})\big\|_{L^r_x}\cdot \ft u_0(n)\Big\|_{L^q_I\l^2_n} \notag \\
 & \les 
\sqrt{p} 
\jb{T}^\frac{d}{2}  |I|^\frac{1}{q}\| \eta  \|_{H^{\frac{d}{2}}(\R^d)}
\|u_0\|_{L^2(\T^d)}.
\label{Str4}
\end{align}

When $|\xi |\ges 1$, 
it follows from   the triangle inequality: $\jb{n} \leq \jb{\xi} \jb{\xi - n}$
that 
\begin{align}
\bigg|\frac{\sin(t|\xi|)}{ |\xi|}\bigg|
\les \frac{1}{\jb{\xi}}
\leq \frac{\jb{\xi - n}}{\jb {n}}
\label{Str5}
\end{align}

\noi
for all $n \in \Z^d$.
On the other hand, when $|\xi|\ll 1$, 
it follows from \eqref{Str3} that
\begin{align}
\bigg|\frac{\sin(t|\xi|)}{ |\xi|}\bigg|
\les 
\frac{t}{\jb{\xi}}
\leq 
t \frac{\jb{\xi - n}}{\jb{n}}
\label{Str6}
\end{align}

\noi
for all $n \in \Z^d$.
Hence, proceeding as before with  \eqref{Str5}  and \eqref{Str6}, 
we have 
\begin{align}
\Bigg(\E \bigg\|\frac{\sin (t|\nabla|)}{|\nabla|}
\bu_{1, T}^\omega
&  \bigg\|^p_{L^q_t(I; L^r_x(\R^d))}  \Bigg) ^{1/p}
 \les 
\sqrt{p} 
 \Bigg\|  \bigg\|\frac{\sin (t|\nabla|)}{ |\nabla|}\big(\et   e^{inx})\bigg\|_{L^r_x} \cdot \ft u_1(n) \Bigg\|_{L^q_I\l^2_n} \notag \\
 & \les 
\sqrt{p} 
\max(1, b)
 \Bigg\|  \|\jb{\xi-n}\fet(\xi - n)\|_{L^{r'}_x} \cdot \frac{\ft u_1(n)}{\jb{n}} \Bigg\|_{L^q_I\l^2_n} \notag \\
 & \les 
\sqrt{p} 
\max(1, b) \jb{T}^\frac{d}{2} |I|^\frac{1}{q}\| \eta  \|_{H^{\frac{d}{2}+1}(\R^d)}
\|u_1\|_{H^{-1}(\T^d)}.
\label{Str7}
\end{align}

\noi
Then, 
\eqref{Str0} follows from 
\eqref{Str4}, \eqref{Str7}, and 
a standard argument using Chebyshev's inequality.
See \cite{BOP1, Poc} for details.

\smallskip

\noi
$\bullet$ {\bf Case 2:}
Next, we consider the case $r = \infty$.
In this case, given small $s>0$, 
choose $\wt r \gg1$ such that $s\wt r > d$.
Then, by Sobolev embedding theorem, we have
\[\big\|S(t)(\bu_{0, T}^\omega,\bu_{1, T}^\omega)\big\|_{L^q_t(I; L^\infty_x)}
\les  \big\|\jb{\nb}^s S(t)(\bu_{0, T}^\omega,\bu_{1, T}^\omega)\big\|_{L^q_t(I; L^{\wt{r}}_x)}.\]

\noi
We now proceed as in Case 1 with $\wt{r}$ instead of $r$.
With the triangle inequality $\jb{\xi}^s \les \jb{\xi-n}^s \jb{n}^s$, 
we have 
\begin{align}
\Big(\E  \big\|\jb{\nb}^s\cos( & t|\nb|)
 \bu_{0, T}^\omega
\big\|^p_{L^q_t(I; L^{\wt r}_x)}\Big)^{\frac 1p}
  \les \sqrt p
\jb{T}^\frac{d}{2}  |I|^\frac{1}{q}\| \eta  \|_{H^{\frac{d}{2}+s}(\R^d)}
\|u_0\|_{H^s(\T^d)}
\label{Str8}
\end{align}

\noi
and
\begin{align}
\Bigg(\E \bigg\|\jb{\nb}^s\frac{\sin (t|\nabla|)}{|\nabla|}
& \bu_{1, T}^\omega
 \bigg\|^p_{L^q_t(I; L^{\wt r}_x(\T^d))}  \Bigg) ^{1/p} \notag \\
 & \les \sqrt p 
\max(1, b^2)\jb{T}^\frac{d}{2}  |I|^\frac{1}{q}\| \eta  \|_{H^{\frac{d}{2}+1+s}(\R^d)}
\|u_0\|_{H^{s-1}(\T^d)}.
\label{Str9}
\end{align}

\noi
Once again, 
\eqref{Str0} follows from 
\eqref{Str8}, \eqref{Str9} and 
a standard argument using Chebyshev's inequality.
\end{proof}

Next, we prove  a probabilistic estimate involving the 
$L^\infty_t$-norm.
This proposition replaces Proposition 3.3 in \cite{OP}
and plays an
important role in treating the three-dimensional case.
Define an operator $\wt S(t)$ on a pair $(f_0, f_1)$ of functions on $\R^d$ by 
\begin{equation}
 \wt S(t)(f_0, f_{1}) 
:=  -\frac{|\nb|}{\jb{\nb}}\sin (t|\nb|) f_{0} + \frac{ \cos (t|\nb|)}{\jb{\nb}}f_{1}.
\label{wlinear}
\end{equation}

\noi
Namely, we have $\dt   S(t)(f_0, f_{1}) =  \jb{\nb}\wt S(t)(f_0, f_{1}) $.

\begin{proposition}\label{PROP:infty}
Let $T\geq 1$.  Given  a pair   $(u_0, u_1)$ of real-valued functions defined on $\T^d$,
let $(\bu_{0, T}^{\omega}, \bu_{1, T}^\omega)$ be the randomization on $\R^d$ defined in \eqref{R3}, satisfying \eqref{cond}. 
Let $S^*(t) = S(t)$ or $\wt S(t)$ defined in \eqref{Zlinear} and \eqref{wlinear}, 
respectively.
Then, for  $2\leq r\leq\infty$, we have 
\begin{align}
P\Big( \|S^*(t) (\bu_{0, T}^\o, \bu_{1, T}^\o) & \|_{L^\infty_t([0, T];  L^r_x(\R^d))} > \ld\Big)\notag \\
 &  \leq C 
\jb{T} \exp \Bigg( -c \frac{\ld^2}{\jb{T}^{d+2} \|(u_0, u_1)\|^2_{\mathcal{H}^\eps (\T^d)}}\Bigg)
\label{infty00}
\end{align}

\noi
for any $\eps > 0$, 
where the constants $C$ and $c$ depend only on $r$ and $\eps$.

\end{proposition}

Proposition \ref{PROP:infty}  follows as a corollary to the following lemma.
Let  $S_+(t)$ and $S_-(t)$ be the linear propagators for
the half wave equations on $\R^d$ defined by 
\[ S_{\pm}(t) f := 
\mathcal{F}_{\R^d}^{-1} \big(e^{\pm i|\xi|t} \wt f (\xi)\big).\]

\noi
Given $\phi \in H^s(\T^d)$ and $T\geq 1$, 
we define its randomization  $ 
\pphi_T^\o$  on $\R^d$ by 
\begin{align*}
\pphi_T^\omega := 
 \sum_{n \in \Z^d} \et (x) g_{n, 0}(\o) \ft \phi(n) e^{inx},
\end{align*}

\noi
as in  the first component of \eqref{R3}.
Then, we have the following tail estimate on 
the size of 
$S_\pm(t) \pphi^\o_T$
over a time interval of length 1.

\begin{lemma}\label{LEM:infty}
Let $T \geq 1$ and   $2\leq r \leq \infty$.
Given 
 any $\eps > 0$, 
there exist constants $C, c>0$, 
depending only on $r$ and $\eps$, 
such that
\begin{align}
&  P\Big( \|S_\pm(t) \pphi_T^\o\|_{L^\infty_t([j, j+1];  L^r_x(\R^d))} > \ld\Big)
  \leq C \exp \Bigg( -c \frac{\ld^2}{\jb{T}^d\|\phi\|^2_{H^\eps (\T^d)}}\Bigg), 
 \label{infty0a}\\
&  P\Big( \|\jb{\nb}^{-1}S_\pm(t) \pphi_T^\o\|_{L^\infty_t([j, j+1];  L^r_x(\R^d))} > \ld\Big)
  \leq C \exp \Bigg( -c \frac{\ld^2}{\jb{T}^d\|\phi\|^2_{H^{\eps-1} (\T^d)}}\Bigg), 
\label{infty0c}\\
&  P\Bigg( \bigg\|\frac{\sin(t |\nb|)}{|\nb|} \pphi_T^\o   \bigg\|_{L^\infty_t([j, j+1];  L^r_x(\R^d))}  > \ld\Bigg) \notag \\
& \hphantom{XXXXXXXXXXXXXXX}
 \leq C \exp \Bigg( -c \frac{\ld^2}{\max(1, j^2)\jb{T}^d\|\phi\|^2_{H^{\eps-1} (\T^d)}}\Bigg)
 \label{infty0b}
\end{align}

\noi
for any $[j, j+1] \subset [0, T]$.

\end{lemma}

Assuming Lemma \ref{LEM:infty}, 
we first present the proof of Proposition \ref{PROP:infty}.

\begin{proof}[Proof of Proposition \ref{PROP:infty}]

We first consider the case $S^*(t) = S(t)$ and $T\geq 1$.
By subadditivity, \eqref{infty0a}, and \eqref{infty0b}, we have
\begin{align*}
P\Big( \|S(t)  (\bu_{0, T}^\o, \bu_{1, T}^\o)& \|_{L^\infty_t([0, T];  L^r_x(\R^d))} > \ld\Big) \\
&  \leq 
P\Big(\max_{j = 0, \dots, [T]}\|S(t) (\bu_{0, T}^\o, \bu_{1, T}^\o)\|_{L_t^\infty([j, j+1]; L^r_x(\R^d))} > \ld\Big)\\
& \leq 
\sum_{j = 0}^{[T]}
P\Big(\|S(t) (\bu_{0, T}^\o, \bu_{1, T}^\o)\|_{L_t^\infty([j, j+1]; L^r_x(\R^d))} > \ld\Big)\\
& \leq 
\sum_{j = 0}^{[T]}
P\bigg(\|\cos (t|\nb|) \bu_{0, T}^\o\|_{L_t^\infty([j, j+1]; L^r_x(\R^d))} > \frac{\ld}{2}\bigg)\\
& \hphantom{XXX}
+
\sum_{j = 0}^{[T]}
P\Bigg(\bigg\|\frac{\sin (t|\nb|)}{|\nb|} \bu_{1, T}^\o\bigg\|_{L^\infty([j, j+1]; L^r_x(\R^d))} > \frac{\ld}{2}\Bigg)\\
& \leq C\jb{T} \exp \Bigg( -c \frac{\ld^2}{\jb{T}^{d+2} \|(u_0, u_1) \|^2_{\mathcal{H}^\eps(\T^d)}}\Bigg).
\end{align*}

\noi
When $S^*(t) = \wt S(t)$, \eqref{infty00} 
follows from \eqref{infty0a} and \eqref{infty0c}.
In this case, we obtain 
$\jb{T}^d$ instead of $\jb{T}^{d+2}$ on the right-hand side of 
\eqref{infty00}.
\end{proof}

Finally, we present the proof of Lemma \ref{LEM:infty}.

\begin{proof}[Proof of Lemma \ref{LEM:infty}]
 We first prove \eqref{infty0a}.
 In the following, we only consider the case of $S_+(t)$.
Set
$\bz^\o(t) = \bz_T^\o(t): = S_+(t) \pphi_T^\o.$

\smallskip

\noi
{\bf Part 1\,(a):}
We first consider the case $r < \infty$.
The first half of the reduction (up to \eqref{infty4}) is exactly the same as that
in the proof of Lemma 3.4 in \cite{OP}.
We decided to include it for reader's convenience. 
Without loss of generality, assume $j = 0$. 
For $k \in \mathbb{N} \cup \{0\}$, 
let $\{ t_{\l, k}:\l = 0, 1,\dots, 2^k\} $ be $2^{k}+1$ equally spaced points on $[0, 1]$, 
i.e.~$t_{0,k}=0$ and $t_{\l, k} - t_{\l-1, k} = 2^{-k}$ for $\l =1, \dots, 2^k$.
Then, given $t \in [0, 1]$, 
we have
\begin{align}
  \bz^\o(t)
= \sum_{k = 1}^\infty \big(\bz^\o(t_{\l_k, k}) - \bz^\o(t_{\l_{k-1}, k-1})\big)
+ \bz^\o(0)
\label{infty2}
\end{align}

\noi
for some $\l_k = \l_k(t) \in \{0, \dots, 2^k\}$.

By the square function estimate and Minkowski's integral inequality
with \eqref{infty2}, we have
\begin{align*}
 \| \bz^\o& \|_{L^\infty_t([0, 1]; L^r_x(\R^d))}  \notag\\
& \les
\bigg(\sum_{\substack{N \geq 1\\\text{dyadic}}}
\Big( \sum_{k =1}^\infty
 \max_{0\leq \l_k \leq 2^k} 
\big\|\P_N 
 \big(\bz^\o(t_{\l_k, k}) - \bz^\o(t_{\l'_{k-1}, k-1})\big)
\big\|_{L^r_x(\R^d)} \Big)^2\bigg)^\frac{1}{2}
+ \|\bz^\o(0)\|_{L^r_x(\R^d)},
 \end{align*}

\noi
where $t_{\l'_{k-1}, k-1}$ is one of the $2^{(k-1)}$+1 equally spaced points such that 
\begin{align}
|t_{\l_k, k} - t_{\l'_{k-1}, k-1}| \leq 2^{-k}.
\label{infty3a}
\end{align}

\noi
Hence, for  $p \geq 2$, we have
\begin{align}
\Big(\E\big[  \| \bz^\o& \|_{L^\infty_t([0, 1]; L^r_x(\R^d))}\big]^p\Big)^\frac{1}{p} \notag \\
& \les
\Bigg(\sum_{\substack{N \geq 1\\\text{dyadic}}}
\bigg( \sum_{k = 1}^\infty
\Big(\E \Big[
 \max_{0\leq \l_k \leq 2^k} 
 \big\|\P_N 
 \big(\bz^\o(t_{\l_k, k}) - \bz_\pm^\o(t_{\l'_{k-1}, k-1})\big)
 \big\|_{L^r_x(\R^d)}
 \Big]^p
  \Big)^\frac{1}{p}\bigg)^2\Bigg)^\frac{1}{2}\notag \\
& + 
\Big(\E\big[ \|\bz^\o(0)\|_{L^r_x(\R^d)}\big]^p\Big)^\frac{1}{p}.
\label{infty4}
\end{align}

Proceeding as in \eqref{Str4} with Lemma \ref{LEM:HC} and \eqref{Str2}, 
 the second term on the right-hand side of \eqref{infty4} can be bounded by
\begin{align}
\Big(\E\big[ \|\bz^\o(0)\|_{L^r_x(\R^d)}\big]^p\Big)^\frac{1}{p}
& \leq \Big\|\|\pphi_T^\o\|_{L^p(\O)}\Big\|_{L^r_x}
\les \sqrt p 
\Big\|\|\et (x)  \ft \phi(n) e^{inx}\|_{\l^2_n}\Big\|_{L^r_x}\notag\\
& \leq \sqrt p 
\Big\|\|\et (x)   e^{inx}\|_{L^r_x} \cdot \ft \phi(n)\Big\|_{\l^2_n} \notag\\
& \les \sqrt p \jb{T}^\frac{d}{2} \|\eta\|_{H^\frac{d}{2}(\R^d)}
\|\phi \|_{L^2_x(\T^d)}
\label{infty5}
\end{align}

\noi
for $p \geq r\geq 2$. 
	
In the following, we first estimate	
\begin{align*}
I_N: =  \sum_{k = 1}^\infty
\bigg(\E \Big[
 \max_{0\leq \l_k \leq 2^k} 
 \big\|\P_N 
 \big(\bz^\o(t_{\l_k, k}) - \bz^\o(t_{\l'_{k-1}, k-1})\big)
 \big\|_{L^r_x(\R^d)}
 \Big]^p
  \bigg)^\frac{1}{p}
\end{align*}
	
\noi
for each  dyadic $N \geq 1$.
Let 
\begin{align} q_k  := \max (\log 2^k, p, r)
\sim \log 2^k + p + r.
\label{infty5a}
\end{align}

\noi
Then, we have
\begin{align}
I_N 
& \leq  \sum_{k = 1}^\infty
\bigg(
 \sum_{\l_k = 0}^{2^k} 
 \E 
 \big\|\P_N 
 \big(\bz^\o(t_{\l_k, k}) - \bz^\o(t_{\l'_{k-1}, k-1})\big)
 \big\|_{L^r_x}^{q_k}
  \bigg)^\frac{1}{q_k}\notag\\
\intertext{Noting  that $(2^k+1)^\frac{1}{q_k} \les 1$
and applying  Lemma \ref{LEM:HC}, 
}
& \les  \sum_{k = 1}^\infty
 \max_{0\leq \l_k \leq 2^k} 
\Big(\E
 \big\|\P_N 
 \big(\bz^\o(t_{\l_k, k}) - \bz^\o(t_{\l'_{k-1}, k-1})\big)
 \big\|_{L^r_x}^{q_k}
  \Big)^\frac{1}{q_k} \notag\\
& \les  \sum_{k = 1}^\infty
\sqrt{q_k}
 \max_{0\leq \l_k \leq 2^k} 
\Big\| 
 \big\| 
\P_N \big( S_+(t_{\l_k, k}) - S_+(t_{\l'_{k-1}, k-1})\big)
( \et e^{inx})
 \big\|_{L^r_x}\cdot \ft \phi(n)
 \Big\|_{\l^2_n}. 
 \label{infty5b}
 \end{align}

For $|\xi|\sim N$, it follows from \eqref{infty3a} that 
\begin{align}
 \Big|e^{i|\xi| t_{\l_k, k} }
-  e^{i|\xi|  t_{\l'_{k-1}, k-1}}\Big| \les \min(1, 2^{-k }N).
\label{infty5c}
\end{align}

\noi
We now proceed as in  \eqref{Str2}.
 With \eqref{infty5c} and the triangle inequality, we have
\begin{align}
 \big\| 
\P_N \big(  S_+(t_{\l_k, k})   - S_+(  & t_{\l'_{k-1}, k-1})\big)
 ( \et e^{inx}) \big\|_{L^r_x(\R^d)}
 \les
\min(1, 2^{-k }N)
\|   \fet(\xi-n)\|_{L^{r'}_{|\xi|\sim N}(\R^d)} \notag \\
&  \les 
\jb{n}^\eps N^{-\eps} \min(1, 2^{-k }N)
\| \jb{\xi - n}^{\be+\eps}   \fet(\xi-n)\|_{L^{2}_\xi(\R^d)}\notag \\
& \les\jb{n}^\eps N^{-\eps} \min(1, 2^{-k }N)
 \| \et  \|_{H^{\frac{d}{2}}(\R^d)} \notag\\
& \les \jb{T}^\frac{d}{2} 
\jb{n}^\eps N^{-\eps} \min(1, 2^{-k }N)
\| \eta  \|_{H^{\frac{d}{2}}(\R^d)} 
\label{infty5d}
\end{align}

\noi
as long as 
$\eps > 0$ is sufficiently small such that 
$\beta +\eps = \frac{(r-2)}{2r} d+2 \eps \leq \frac d2 $.
Hence, from \eqref{infty5b} and \eqref{infty5d}, we obtain
 \begin{align}
I_N & \les  \jb{T}^\frac{d}{2} \|\eta\|_{H^\frac{d}{2}(\R^d)}
\sum_{k = 1}^\infty
\sqrt{q_k}
N^{-\eps}\min(1, 2^{-k }N) \|\phi\|_{H^\eps(\T^d)}.
\label{infty6}
\end{align}

Separating the summation (in $k$) into $2^{-k}N \geq 1$
and $2^{-k}N <  1$
and applying \eqref{infty5a}, 
we have
\begin{align}\sum_{k = 1}^\infty
\sqrt{q_k}
N^{-\eps}\min(1, 2^{-k }N) \leq C_{r, \eps} \sqrt p N^{-\frac{\eps}{2}}.
\label{infty6a}
\end{align}

\noi
See \cite{OP} for details.
Finally, putting \eqref{infty4},  \eqref{infty5},  \eqref{infty6}, and \eqref{infty6a},  
together, we obtain 
\begin{align*}
\Big(\E\big[  \| \bz^\o \|_{L^\infty_t([0, 1]; L^r_x(\R^d) )}\big]^p\Big)^\frac{1}{p}
 \leq C_{r, \eps}
\sqrt{p}
  \jb{T}^\frac{d}{2} \|\eta\|_{H^\frac{d}{2}(\R^d)}
 \|\phi\|_{H^\eps(\T^d)}
\end{align*}

\noi
for all $p \geq r$ and  sufficiently small $\eps > 0$.
The rest follows from a standard argument
using Chebyshev's inequality.

\smallskip

\noi
{\bf Part 1\,(b):}
Next, we consider the case $r = \infty$.
It follows from Sobolev embedding that, 
given any small $\eps>0$,  there exists  large $\tilde r \gg1$ with $\eps \wt{r} > d$
such that 
\begin{align*}
P\Big( \|S_\pm(t)  \pphi_T^\o\|_{L^\infty_t([j, j+1];  L^\infty_x(\R^d))} > \ld\Big)
 \leq P\Big( \|\jb{\nb}^\eps S_\pm(t) \pphi_T^\o\|_{L^\infty_t([j, j+1];  L^{\wt{r}}_x(\R^d))} > C \ld\Big).
\end{align*}

\noi
Then, the rest follows from the triangle inequality $\jb{\xi}^\eps \les \jb{n}^\eps \jb{\xi-n}^\eps$
and the argument in Part 1\,(a).

\smallskip

\noi
{\bf Part 2:} 
Next, we consider \eqref{infty0c}.
By proceeding as in Part 1\,(a), 
the only essential modifications appear only in \eqref{infty5}
and \eqref{infty5d}.
With \eqref{infty5c} and the triangle inequality: $\jb{\xi}^{-1} \leq \jb{n}^{-1} \jb{\xi-n}$, we have
\begin{align*}
 \big\| 
\P_N \jb{\nb}^{-1} \big(  S_+(t_{\l_k, k})   - S_+(  & t_{\l'_{k-1}, k-1})\big)
 ( \et e^{inx}) \big\|_{L^r_x(\R^d)} \notag\\
&  \les
 \jb{n}^{-1}
\min(1, 2^{-k }N)
\|\jb{\xi-n}   \fet(\xi-n)\|_{L^{r'}_{|\xi|\sim N}(\R^d)} \notag \\
&  \les 
\jb{n}^{\eps-1} N^{-\eps} \min(1, 2^{-k }N)
\| \jb{\xi - n}^{\be+\eps+1}   \fet(\xi-n)\|_{L^{2}_\xi(\R^d)}\notag \\
& \les \jb{T}^\frac{d}{2} 
\jb{n}^{\eps-1} N^{-\eps} \min(1, 2^{-k }N)
\| \eta  \|_{H^{\frac{d}{2}+1}(\R^d)} .
\end{align*}

\noi
This shows how one modifies \eqref{infty5d}, while \eqref{infty5} can be modified similarly.
Then, the rest follows as in Part 1\,(a), yielding \eqref{infty0c}.

\smallskip

\noi
{\bf Part 3:} 
Finally, we prove \eqref{infty0b} when $r < \infty$.
The modification needed for the case $ r = \infty$
is straightforward  as in Part 1\,(b).
Define $\bZ^\o(t)$ by 
\[\bZ^\o(t): = \frac{\sin(t |\nb|)}{|\nb|} \pphi_T^\o .\]

\noi
Repeating the argument in Part 1\,(a)
(but on $[j, j + 1]$ instead of $[0, 1]$), we have 
\begin{align*}
\Big(\E\big[  \| \bZ^\o& \|_{L^\infty_t([j, j+1]; L^r_x(\R^d))}\big]^p\Big)^\frac{1}{p} \notag \\
& \les
\Bigg(\sum_{\substack{N \geq 1\\\text{dyadic}}}
\bigg( \sum_{k = 1}^\infty
\Big(\E \Big[
 \max_{0\leq \l_k \leq 2^k} 
 \big\|\P_N 
 \big(\bZ^\o(t_{\l_k, k}) -\bZ^\o(t_{\l'_{k-1}, k-1})\big)
 \big\|_{L^r_x(\R^d)}
 \Big]^p
  \Big)^\frac{1}{p}\bigg)^2\Bigg)^\frac{1}{2}\notag \\
& + 
\Big(\E\big[ \|\bZ^\o(j)\|_{L^r_x(\R^d)}\big]^p\Big)^\frac{1}{p}
= : \I + \II.
\end{align*}

\noi
When $j = 0$, then we have $\II = 0$.
When $j \geq 1$, proceeding as in \eqref{Str7}, we have 
\begin{align}
\II \les \sqrt p \cdot j 
  \jb{T}^\frac{d}{2} \|\eta\|_{H^{\frac{d}{2}+1}(\R^d)}
\|\phi\|_{H^{-1}(\T^d)}
\label{infty10a}
\end{align}

\noi
for $p \geq r$. 
As for $\I$, we simply repeat the computations in Part 1\,(a)
with a modification in \eqref{infty5d}.
For non-zero $|\xi|\sim N$, it follows from  Mean Value Theorem with \eqref{infty3a} 
and the triangle inequality that 
\begin{align}
 \Bigg|\frac{\sin (t_{\l_k, k} |\xi| )
- \sin(  t_{\l'_{k-1}, k-1}|\xi|)}{|\xi|}\Bigg| \les \min(1 , 2^{-k }N)\frac{\jb{\xi - n}}{\jb{n}}.
\label{infty10c}
\end{align}

\noi
Proceeding as in Part 1 with \eqref{infty10c}, 
we obtain 
\begin{align}
\I 
 \leq C_{r, \eps}
\sqrt{p}
  \jb{T}^\frac{d}{2} \|\eta\|_{H^{\frac{d}{2}+1}(\R^d)}
 \|\phi\|_{H^{\eps-1}(\T^d)}.
\label{infty10b}
\end{align}

\noi
Then, the desired estimate \eqref{infty0b} follows
from \eqref{infty10a} and \eqref{infty10b}.
\end{proof}

\section{Proof of Proposition \ref{PROP:GWP2}}
\label{SEC:4d}

In this section,  we present the proof of  Proposition \ref{PROP:GWP2}. 
In Subsection \ref{SUBSEC:HIGH}, 
we treat the higher dimensional case $ d = 4, 5$.
Then, we briefly discuss 
some components of the proof for the $d = 3$ case
in Subsection \ref{SUBSEC:3d}.

\subsection{Higher dimensional case}\label{SUBSEC:HIGH}
In this subsection, we consider the case $d = 4, 5$.
In this case, the following probabilistic a priori energy bound 
plays 
an essential role, 
replacing Proposition 5.2 in \cite{Poc}.


\begin{lemma}[Probabilistic energy bound]\label{prop:energy}
Let $d=4$ or $5$
and $s <1$ satisfy the condition in Theorem \ref{THM:GWP}.
Given   $(u_0, u_1) \in \mathcal{H}^s(\T^d)$ and $T\geq 1$, 
let $(\bu_{0, T}^\o, \bu_{1, T}^\o)$ be the randomization on $\R^d$ defined in \eqref{R3}, 
satisfying \eqref{cond}. 
Suppose that  $\bv^\omega$ is a solution to the 
Cauchy problem \eqref{NLW4} on $[0,T]$. 
Then, given small $\eps > 0$, 
there exists a set $\wt{\Omega}_{T,\eps}\subset\Omega$
with $P(\wt{\Omega}_{T,\eps}^c)<\frac{\eps}{2}$, such that 
for all $\omega\in\wt{\Omega}_{T,\eps}$,
we have 
\begin{align}
\sup_{t \in [0, T]}E(\bv^\omega(t))\leq C\Big(T,\eps,\|(u_0,u_1)\|_{\H^s(\T^d)}\Big),
\label{P1}
\end{align}

\noi
and thus also
\begin{equation*}
\big\|(\bv^\omega, \pa_t \bv^\omega)\big\|_{L^\infty_t([0,T]; \H^1(\R^d))}
\leq C_0\Big(T,\eps,\|(u_0,u_1)\|_{\H^s(\T^d)}\Big).
\end{equation*}

\end{lemma}

\begin{proof}
The proof of this lemma follows closely the proof of Proposition 5.2 in \cite{Poc}
and thus we only sketch the proof of \eqref{P1} when $d = 4$.

Taking the time derivative of the energy $E(v^\o(t))$ with \eqref{NLW4}
and 
integrating by parts, we have
\begin{align*}
\frac{d}{dt}E(\bv^\omega(t))
&=\int_{\R^4}\pa_t \bv^\omega\Big(\pa_{t}^2\bv^\omega-\Delta \bv^\omega+(\bv^\omega)^3\Big)dx
=\int_{\R^4}\pa_t \bv^\omega\Big((\bv^\omega)^3-(\bv^\o+ \bz^\o)^3\Big)dx.
\end{align*}

\noi
By H\"older's inequality, we have
\begin{align*}
\bigg|\frac{d}{dt}E(\bv^\omega(t))\bigg|
&\leq C\Big(E(\bv^\omega(t))\Big)^{\frac12}
\Big(\|\bz^\omega\|_{L^6_x(\R^4)}^3+\|\bz^\omega\|_{L^\infty_x(\R^4)}\|\bv^\omega\|_{L^4_x(\R^4)}^2\Big).
\end{align*}

\noi
Noting that $E(v^\o(0)) = 0$, integration in time  then yields
\begin{align*}
\Big(E(\bv^\omega(t))\Big)^{\frac12}
&\leq C \|\bz^\omega\|_{L^3_TL^6_x}^3+C\int_0^t\|\bz^\omega(t')\|_{L^\infty_x}
\Big(E(\bv^\omega(t'))\Big)^{\frac12}dt'.
\end{align*}

\noi
By Gronwall's inequality,
we obtain 
\begin{align}\label{energy_estim}
\sup_{t \in [0, T]}\Big(E(\bv^\omega(t))\Big)^{\frac12}
&\leq C \|\bz^\omega\|_{L^3_TL^6_x}^3 e^{C\|\bz^\omega\|_{L^1_TL_x^\infty}}.
\end{align}

\noi
Then, 
by choosing  
 $\ld = K  \jb{T}^{4} \|(u_0, u_1)\|_{\H^{s}(\T^4)}$
and $K = K(\eps) \gg 1$, 
it follows 
from Proposition \ref{PROP:Str}
that there exists
$\wt{\Omega}_{T,\eps}\subset\Omega$
with $P(\wt{\Omega}_{T,\eps}^c)<\frac{\eps}{2}$ such that 
for all $\omega\in\wt{\Omega}_{T,\eps}$, we have
\begin{align}\label{zbdd}
 \|\bz^\omega\|_{L^3_TL^6_x} +  \|\bz^\omega\|_{L^1_TL^\infty_x}
 \leq K\jb{T}^4\|(u_0,u_1)\|_{\mathcal{H}^s(\T^d)}.
\end{align}
Combining this with \eqref{energy_estim} yields \eqref{P1}.
\end{proof}

The key deterministic ingredient 
in the proof of Proposition \ref{PROP:GWP2}
is the following
``good'' local well-posedness result
of the perturbed NLW \eqref{NLW5}.
In particular, 
the  time of local existence is characterized
only in terms of 
the $\dot {\mathcal{H}}^1$-norm of the initial data $(v_0,v_1)$
and the size of the perturbation $f$.

\begin{lemma}[Proposition 4.3 in \cite{Poc}]
\label{LEM:pLWP}
Let $d = 4$ or $5$
and  $(\bv_0,\bv_1)\in\mathcal{\dot H}^1(\R^d)$. 
Then, there exists a function 
$\tau : [0, \infty) \times \R_+\times \R_+ \to \R_+$,
non-increasing in the first two arguments, such that 
if  $f$ satisfies the condition
\begin{equation}\label{f}
\|f\|_{L^\frac{d+2}{d-2}_tL^{\frac{2(d+2)}{d-2}}_x([t_0,t_0+\tau_\ast] \times\R^d)}\leq K\tau_\ast^\theta
\end{equation}

\noi
for some $K,\theta>0$ and $\tau_\ast\leq \tau=\tau\big(\|(\bv_0,\bv_1)\|_{\dot{\mathcal{H}}^1(\R^d)}, K,\theta\big)
\ll1$, 
then  
there exists  a unique solution $(\bv, \dt \bv) \in C([t_0,t_0+\tau_\ast]; \dot{\mathcal{H}}^1(\R^d))$
to \eqref{NLW5}.
Moreover,
\begin{align}\label{Stricv}
\|\bv\|_{L^q_t([t_0,t_0+t_\ast]; L^r_x(\R^d))}\leq C(\|(\bv_0,\bv_1)\|_{\mathcal{\dot H}^1(\R^d)}),
\end{align}
for all $1$-admissible pairs $(q,r)$.
 \end{lemma}

Now, we are ready to present the proof of 
Proposition \ref{PROP:GWP2} for $ d= 4, 5$.

\begin{proof}[Proof of Proposition \ref{PROP:GWP2}]

Let $T\geq 1$ and $\eps > 0$.
Given $(\bu_{0, T}^\omega, \bu_{1, T}^\omega)$, 
let $\bz^\o$ and $\bv^\o$ be as in \eqref{z}.
By  Lemma \ref{prop:energy}, 
there exists a set 
 $\O_1 $ with 
\begin{align}
P(\Omega_1^{c})<\frac{\eps}{2}
\label{Q3} 
\end{align}

\noi
such that 
\begin{align}
\sup_{t \in [0, T]}\| (\bv^\o  (t), \dt \bv^\o(t)) \|_{ \mathcal{H}^1(\R^d)} 
\leq C_0 : = C_0(T, \eps,  \|(u_0, u_1)\|_{{\mathcal{H}^s(\T^d)}}) < \infty,
\label{Q4}
\end{align}

\noi
 for each $\o \in \O_1$.

Let  $\tau = \tau\big(C_0, K, \theta)$  be as  in Lemma \ref{LEM:pLWP}, 
where   $K =  \| (u_0, u_1)\|_{\mathcal{H}^{0}(\T^d)}$
and $\theta = \frac{d-2}{2(d+2)}$.
Fix $\tau_* \leq \tau$ 
to be chosen later. 
By writing 
 $[0, T] = \bigcup_{j = 0}^{[ T/\tau_*]}  I_j$
with   $ I_j = [ j \tau_*, (j+1) \tau_*]\cap [0, T]$, 
define $ \O_2$ by 
\begin{align}
\O_2 := \Big\{ \o \in \O: \, 
\| \bz^\o \|_{L^\frac{d+2}{d-2}_{ I_j} L^{\frac{2(d+2)}{d-2}}_x} \leq K | I_j|^{\theta}, 
j = 0, \dots, \big[\tfrac{T}{\tau_*}\big]\Big\}.
\label{Q4a}
\end{align}

\noi
Then, by Proposition \ref{PROP:Str} with $|I_j | \leq \tau_*$, 
we have 
\begin{align*}
P(\O_2^c) 
& \leq 
\sum_{j = 0}^{[\frac{T}{\tau_*}]}
P \Big( \| \bz^\o \|_{L^\frac{d+2}{d-2}_{I_j} L^{\frac{2(d+2)}{d-2}}_x} > K |I_j |^{\theta}\Big)
 \les \frac{T}{\tau_*} \exp \Bigg( - \frac{c}{\jb{T}^{d+2} \tau_*^{2\theta}}\Bigg).
\intertext{By making $\tau_*$ smaller if necessary,}
& \les \frac{T}{\tau_*} \tau_* \exp \Bigg( - \frac{c}{2\jb{T}^{d+2} \tau_*^{2\theta}}\Bigg)
 =  T \exp \Bigg( - \frac{c}{2\jb{T}^{d+2} \tau_*^{2\theta}}\Bigg).
\end{align*}

\noi
Hence, by choosing $\tau_* = \tau_*(T, \eps)$ 
sufficiently small, 
we conclude that 
\begin{align}
P(\Omega_2^{c})<\frac{\eps}{2}.
\label{Q5} 
\end{align}

Let $\wt \O_{T, \eps}: = \O_1 \cap \O_2$.
Then, from  \eqref{Q3} and \eqref{Q5}, we have 
$P(\wt \Omega_{T, \eps}^{c})<\eps.$
Moreover, 
it follows from 
Lemma \ref{LEM:pLWP}
applied iteratively
with \eqref{Q4} and 
\eqref{Q4a}
 on the intervals $I_j$,
$j=0,\dots, [\frac{T}{\tau_*}]$,
that for each $\o \in \wt \O_{T, \eps}$, 
there exists a unique solution $\bv^\o$ to \eqref{NLW4} 
on $[0, T]$.
Hence, for $\o \in \wt \O_{T, \eps}$, 
there exists a unique solution $\bu^\o = \bz^\o + \bv^\o$ to \eqref{NLW3} on $[0, T]$.
Moreover, \eqref{Stricbv}
follows from \eqref{Stricv}.
\end{proof}

\subsection{Three-dimensional case}
\label{SUBSEC:3d}

In the following, we briefly sketch the idea of the proof of Proposition \ref{PROP:GWP2}
when $ d= 3$.  
In this case, 
the additional difficulty comes from the lack of a probabilistic a priori energy bound (Lemma \ref{prop:energy}).
Therefore, as in \cite{OP}, we need to establish a uniform probabilistic energy bound 
for approximating random solutions.

Let   $(u_0, u_1)\in \mathcal{H}^s(\T^3)$ with $\frac 12 < s< 1$ and $T\geq 1$.
Given $N \geq 1$ dyadic, 
define $\bu_{j,T,  N}^\o$, $j = 0, 1$,  by 
\begin{equation}
   \bu_{j, T, N}^\o := \P_{\leq N} \bu_{j, T}^\o.
\label{cutoffdata}
\end{equation}

\noi
Let $\bu_N$ be the smooth global solution to \eqref{NLW3} on $\R^3$
with initial data
$(\bu_N, \dt \bu_N)|_{t = 0} = (\bu_{0,T,  N}^\o, \bu_{1, T, N}^\o) \in \mathcal{H}^\infty(\R^3)$. 
Denote the linear and nonlinear parts of $\bu_N$
by 
$ \bz_N = \bz_N^\o$ and $\bv_N = \bv_N^\o$, respectively.
In particular, $\bv_N$ is the 
smooth global solution
to the following perturbed NLW on $\R^3$:
\begin{align}
\begin{cases}
 \dt^2 \bv_N - \Dl \bv_N + ( \bv_N+\bz_N)^5 = 0,\\
 (\bv_N, \dt \bv_N) |_{t = 0} = (0, 0).
\end{cases}
\label{pNLW1}
\end{align}

\noi
While 
we have $\| (\bv_N^\o,\partial_t \bv_N^\o) \|_{L_t^\infty(\R;  \dot{ \mathcal{H}}^1(\R^3))} \leq C(N, \o)  < \infty$
for each $N \in \mathbb{N}$, 
there is no uniform (in $N$) control on the $H^1$-norm of $\bv_N$.
The following lemma 
establishes a uniform (in $N$) bound on the $H^1$-norm of $\bv_N$
in a probabilistic manner.

\begin{lemma}\label{PROP:Penergy}
Let $ s \in (\frac 12, 1)$ and $N \geq 1$ dyadic.
Given $T,  \eps > 0$, there exists $\wt{ \O}_{N, T, \eps}\subset \O$
such that 
\begin{itemize}
\item[(i)] $P(\wt \O_{N, T, \eps}^c) < \eps$, 
\item[(ii)]  
There exists a finite constant $C(T, \eps, \|(u_0, u_1)\|_{{\mathcal{H}}^s(\T^3)}) > 0$
such that
the following energy bound holds:
\begin{align}
\sup_{t \in [0, T]}\| (\bv^\o_N  (t),\partial_t \bv^\o_N  (t)) \|_{ \mathcal{H}^1(\R^3)} 
\leq C(T, \eps,  \|(u_0, u_1)\|_{{\mathcal{H}}^s(\T^3)}),
\label{Penergy1}
\end{align}

\noi
for all   solutions $\bv^\o_N$ to \eqref{pNLW1} 
on $[0, T]$ with  $\o \in \wt{\O}_{N, T, \eps}$.
\end{itemize}
	
\noi	
Note that  the constant $C(T, \eps,  \|(u_0, u_1)\|_{{\mathcal{H}}^s(\R^3)})$ 
is independent of dyadic $N \geq 1$.
\end{lemma}

Lemma \ref{PROP:Penergy} plays the role of 
Proposition 4.1 in \cite{OP}
and is a suitable substitute of  the probabilistic a priori energy estimate (Lemma \ref{prop:energy})
when $ d= 3$.
One can prove 
Lemma \ref{PROP:Penergy} 
exactly in the same manner as 
Proposition 4.1 in \cite{OP}, 
by simply replacing the probabilistic Strichartz estimates on $\R^d$
(Lemma 3.2 and Proposition 3.3 in \cite{OP})
with 
the appropriate probabilistic Strichartz estimates for our problem
(Propositions \ref{PROP:Str} and  \ref{PROP:infty} above).
Therefore, we omit details.

The following lemma is the key deterministic ingredient in this case.
Given $f \in L^5_{t, \text{loc}}L^{10}_x$, 
  let $f_N = \P_{\leq N} f $ for dyadic $N \geq 1$.
Consider the following perturbed NLW:
\begin{align}
\begin{cases}
\pa_{t}^2\bv_N-\Delta \bv_N+(\bv_N+f_N)^5=0\\
(\bv_N, \dt \bv_N)|_{t = 0} = (0, 0).
\end{cases}
\label{pNLW2}
\end{align}

\begin{lemma}[Proposition 5.2 in \cite{OP}]
\label{PROP:pLWP2}
Let $f, f_N$, and $\bv_N$ be as above.
Given finite $T > 0$, assume that the following conditions hold:
\begin{itemize}
\item[\textup{(i)}]
There exist $K, \theta > 0$ such that  
\begin{equation*}
\|f\|_{L^5_tL^{10}_x(I\times \R^3)}\leq K|I |^\theta
\end{equation*}

\noi
for any compact interval $I \subset [0, T]$.	

\item[\textup{(ii)}]
For each dyadic $N\geq 1$, 
a solution $\bv_N$  to \eqref{pNLW2} exists on $[0, T]$
and 
satisfies the following uniform a priori energy bound:
\begin{equation*}
\sup_N \sup_{t \in [0, T]} \|(\bv_N(t),\dt \bv_N(t))\|_{{\mathcal{H}}^1(\R^3)} < C_0(T) < \infty.
\end{equation*}

\noi
\item[\textup{(iii)}]
There exists $\al > 0$ such that 
\begin{equation*}
 \|f - f_N \|_{L^5_T L^{10}_x} < C_1(T) N^{-\al}
\end{equation*}

\noi 
for all dyadic $N \geq 1$.

\end{itemize}

\noi
Then, there exists a unique solution $(\bv, \dt \bv)\in C([0, T]; {\mathcal{H}}^1(\R^3))$
 to \eqref{NLW5}
with $(\bv, \dt \bv)|_{t = 0} = (0, 0)$, 
satisfying  
\begin{equation*}
 \sup_{t \in [0, T]} \|(\bv(t),\dt \bv(t))\|_{{\mathcal{H}}^1(\R^3)} < 2C_0(T) < \infty.
\end{equation*}

\end{lemma}

Finally, 
with Proposition \ref{PROP:Str}, Lemmas \ref{PROP:Penergy} and \ref{PROP:pLWP2}, 
one can prove Proposition \ref{PROP:GWP2}, 
following  
the proof of Proposition 6.1 in \cite{OP}.
Since the argument is identical, 
we omit details.

\appendix

\section{On the finite speed of propagation}
\label{SEC:APP}

In this appendix, we discuss
the issues related to the finite speed of propagation.
In particular, we 
provide  details  of the reduction from 
Proposition \ref{PROP:GWP2} on $\R^d$  to Proposition \ref{PROP:GWP} on $\T^d$.
For simplicity of the presentation, we only consider  the case $d = 4$.

In the following, fix  $(u_0, u_1) \in \mathcal{H}^s(\T^4)$ with $0<s<1$, $T\geq 1$, 
and $\eps > 0$.
 By Proposition \ref{PROP:GWP2}, 
 there exists
$\wt \O_{T, \frac 12 \eps}$ 
with
$P(\wt \O_{T, \frac 12 \eps}^c) < \frac 12 \eps$
and, for each $\o \in \wt \O_{T, \frac 12 \eps}$, there exists a unique solution 
$\bv^\o$ to \eqref{NLW4} on $[0, T]$, 
satisfying the energy bound \eqref{Q4}.

Given $N \in \mathbb{N}$, 
define periodic functions $u_{j,N}^\o$  on $\T^d$, $j = 0, 1$, by
\[u_{j,N}^\o := \P_{\leq N} u_{j}^\o
 = \sum_{|n|\leq N} g_{n, j} (\o) \ft u_j(n) e^{in \cdot x}  \]

\noi
and set  
$(\bu_{0,N,  T}^\o, \bu_{1,N, T}^\o)
= (\et u_{0,N}^\o, \et u_{1,N}^\o)$.
Note that $\bu_{j,N,  T}^\o$ is different from 
 $\bu_{j,T,  N}^\o$
 defined in  \eqref{cutoffdata}.
It follows from an analogue of Lemma \ref{LEM:Hs}
that  
$(\bu_{0,N, T}^\o, \bu_{1, N, T}^\o) \in \mathcal{H}^\infty(\R^4)$
almost surely. 
Therefore, there exists 
a unique (smooth) global solution $\bu_N^\o$ to 
the following Cauchy problem on $\R^4$:
\begin{equation*}
\begin{cases}
\pa_{t}^2 \bu_N^\o -\Delta \bu_N^\o+(\bu_N^\o)^3 =0 
\\
(\bu_N^\o,   \pa_t \bu_N^\o)\big|_{t = 0} = (\bu_{0, N, T}^\o, \bu_{1,N,  T}^\o).
\end{cases}
\end{equation*}

\noi
By the finite speed of propagation (for smooth solutions), 
 $u_N^\o: =  \bu_N^\o |_{[0, T]\times \T^4}$
is a solution to the periodic NLW 
 \eqref{NLW} on the time interval $[0, T]$
with initial data $(u_{0, N}^\o, u_{1, N}^\o)$.

Denote the linear and nonlinear parts of $\bu_N$
by 
\begin{align*}
 \bz_N = \bz_N^\o
 : = S(t) (\bu_{0, N, T}^\o, \bu_{1,N,  T}^\o)
 \qquad \text{and}
 \qquad \bv_N :=\bu_N^\o - \bz_N^\o. 
\end{align*}

\noi
Then, $\bv_N$ is the 
smooth global solution
to the following perturbed NLW on $\R^4$:
\begin{align*}
\begin{cases}
 \dt^2 \bv_N - \Dl \bv_N + (\bv_N+\bz_N)^3 = 0\\
 (\bv_N, \dt \bv_N) |_{t = 0} = (0, 0).
\end{cases}
\end{align*}

\noi
Also, define  $z_{\text{per}, N}^\o$
and $z_{\text{per}}^\o$
by 
\begin{align*}
z_{\text{per}, N}^\o
 : = S_\text{per}(t) (u_{0, N}^\o, u_{1,N}^\o)
  \qquad \text{and}\qquad
 z_{\text{per}}^\o
 : = S_\text{per}(t) (u_{0}^\o, u_{1}^\o).
\end{align*}

\noi
Note that, 
by the finite speed of propagation for the linear solutions, we have
\begin{align}
 \bz_N^\o |_{[0, T]\times \T^4}
  =   z_{\text{per}, N}^\o
  \qquad \text{and}\qquad
   \bz^\o |_{[0, T]\times \T^4}
  =   z_{\text{per}}^\o, 
\label{A3}
\end{align}

\noi
where $\bz^\o$ is as in \eqref{z}.
In particular, $v_N : = \bv_N|_{[0, T]\times \T^4}$ is the 
smooth global solution
to the following perturbed NLW on $\T^4$:
\begin{align}
\begin{cases}
 \dt^2 v_N - \Dl v_N + (v_N+z_{\text{per}, N})^3 = 0,\\
 (v_N, \dt v_N) |_{t = 0} = (0, 0).
\end{cases}
\label{NLW8}
\end{align}

By Proposition \ref{PROP:Str}, we have 
the following probabilistic estimate on $\bz^\o - \bz^\o_N$.
\begin{lemma}\label{LEM:AStr}
Let $T>0$ and $N \in \mathbb N$.
Given $1\leq q< \infty$, $2\leq r \leq  \infty$,
there exist $C,c>0$ such that
\begin{align*}
P\Big(\|\bz^\o - \bz_N^\o& \|_{L^q_t([0, T]; L^r_x(\R^d))}>\ld\Big) 
 \leq C\exp\Bigg(-c\frac{\ld^2}{\jb{T}^{d+2+\frac{2}{q}}
\|\P_{>N} (u_0, u_1)\|_{\mathcal{H}^s(\T^d)}^2}
\Bigg), 
\end{align*}

\noi
provided \textup{(i)} $ s= 0$ if $r < \infty$
and \textup{(i)} $ s> 0$ if $r = \infty$.

\end{lemma}

Noting that 
$\|\P_{>N} (u_0, u_1)\|_{\mathcal{H}^s(\T^4)} \to 0$ as $N \to \infty$, 
given $k \in \mathbb{N}$, there exists $N_k \in \mathbb N$
and $\wt \O_k \subset \O$ with 
$P(\wt \O_k^c) < \frac{\eps}{2^{k+1}}$
such that 
for all $\o\in \wt \O_k$, we have
\begin{align}
\sup_{(q, r) \in \mathcal A} \|\bz^\o - \bz_{N_k}^\o \|_{L^q_T L^r_x}
\leq \frac{1}{k}, 
\label{A4}
\end{align}

\noi
where $\mathcal A =\{ (3, 6), (1, \infty)\}$.
Now, define $\O_{T, \eps}$
by 
\[\O_{T, \eps} =  \wt \O_{T, \frac 12 \eps} \cap \bigg(\bigcap_{k = 1}^\infty \wt \O_k\bigg).\]

\noi
Then, we have $P(\O_{T, \eps}^c) < \eps$.
Recall that $\Omega_{T,\eps}\subset \wt \O_{T,\frac{\eps}{2}}\subset \O_1$,
where $\O_1$ was defined in the proof of Proposition \ref{PROP:GWP2} in Subsection \ref{SUBSEC:HIGH}
such that \eqref{zbdd} and \eqref{Q4} hold for all $\o\in\O_1$. 
Then, by repeating the proof of Lemma \ref{prop:energy} with \eqref{A4}, 
there exists $k_0 \in \mathbb N$ such that 
\begin{align*}
\sup_{t \in [0, T]}\| (\bv_{N_k}^\o  (t), \dt \bv_{N_k}^\o(t)) \|_{ \mathcal{H}^1(\R^4)} 
\leq 2 C_0(T, \eps,  \|(u_0, u_1)\|_{{\mathcal{H}^s(\T^4)}}) < \infty,
\end{align*}

\noi
 for all $\o \in \O_{T, \eps}$
 and all $k \geq k_0$.
Moreover, it follows from \eqref{Stricbv}, \eqref{A3}, 
and the fact that $\Omega_{T,\eps}\subset \wt \O_{T,\frac{\eps}{2}}\subset \O_2$
with $\O_2$ defined in \eqref{Q4a},
that there exists $k_1 \in \mathbb{N}$ such that 
\begin{align}
\| \bv^\o \|_{L^3_TL^6_x(\R^4)}, 
\| \bv_{N_k}^\o \|_{L^3_TL^6_x(\R^4)}
\leq 
C_1(T, \eps,  \|(u_0, u_1)\|_{{\mathcal{H}^s(\T^4)}}) < \infty,
\label{A6}
\end{align}

\noi
 for all $\o \in \O_{T, \eps}$
 and all $k \geq k_1$.

 In the following, we fix  $\o \in \O_{T, \eps}$.
Given an interval $I$, let $X(I) = \{({\bf w},\dt {\bf w}): 
({\bf w},\dt {\bf w})\in  C_I \mathcal{\dot H}^1_x(\R^4), \ 
{\bf w}\in L^3_{I} L^6_x (\R^4) 
\}$.
By Monotone Convergence Theorem with \eqref{A6}, 
we can further subdivide the intervals $I_j$ in \eqref{Q4a}
and relabel them such that
\begin{align}
\| \bv^\o \|_{L^3_{I_j} L^6_x} \leq \g
\label{A7}
\end{align}

\noi
for some sufficiently small $\g > 0$, 
where
 $[0, T] = \bigcup_{j = 0}^{J}  I_j$
 with  $I_j =  [t_j, t_{j+1}]$, $t_0 = 0 < t_1 < \cdots < t_J = T$, 
 and $J < \infty$.
Moreover, it follows from   \eqref{Q4a} and \eqref{A4} that 
there exists $k_2\in \mathbb N$ such that 
\begin{align}
\| \bz^\o \|_{L^3_{I_j} L^6_x}, \| \bz^\o_{N_k} \|_{L^3_{I_j} L^6_x} \leq \g \ll 1, 
\label{A8}
\end{align}

\noi
for all $k \geq k_2$.

Let $k \geq \max(k_0, k_1, k_2)$.
By Monotone Convergence Theorem with \eqref{A6}, we have 
\begin{align}
\| \bv_{N_k}^\o \|_{L^3_t([0, \dl]; L^6_x)}
\leq 4 \g \ll1
\label{A9}
\end{align}

\noi
for some small $\dl  = \dl(k, \o) > 0$ with $[0, \dl] \subset I_0$.
Then, by Lemma \ref{LEM:Strichartz} with \eqref{A7}, \eqref{A8}, and \eqref{A9}, 
we have
\begin{align}
\| \bv^\o - \bv_{N_k}^\o \|_{X([0, \dl])}
\leq \tfrac 12 
\| \bv^\o - \bv_{N_k}^\o \|_{L^3_t([0, \dl]; L^6_x)}
+ 
\tfrac 12 \| \bz^\o - \bz_{N_k}^\o \|_{L^3_t([0, \dl]; L^6_x)}.
\label{A10}
\end{align}
	
\noi
It follows from \eqref{A4}, \eqref{A7}, and \eqref{A10}
that there exists $K_0\geq  \max(k_0, k_1, k_2)$ such that
\begin{align*}
\| \bv_{N_k}^\o \|_{L^3_t([0, \dl]; L^6_x)}
\leq 2 \g 
\end{align*}

\noi
for all $k \geq K_0$.
Then, a continuity argument with  \eqref{A4}, \eqref{A7}, \eqref{A8}, and \eqref{A10}
yields
\begin{align}
\| \bv_{N_k}^\o \|_{L^3_{I_0} L^6_x}\leq 2 \g 
\qquad \text{and}\qquad 
\| \bv^\o - \bv_{N_k}^\o \|_{X({I_0})}
\leq  \| \bz^\o - \bz_{N_k}^\o \|_{L^3_{I_0} L^6_x}
\label{A11}
\end{align}

\noi
for all $k \geq K_0$.

Once again, by Monotone Convergence Theorem with \eqref{A6}, we have 
\begin{align}
\| \bv_{N_k}^\o \|_{L^3_t([t_1 , t_1+ \dl]; L^6_x)}
\leq 4 \g \ll1
\label{A12}
\end{align}

\noi
for some small $\dl  = \dl(k, \o) > 0$ with $[t_1 , t_1 + \dl] \subset I_1$.
By Lemma \ref{LEM:Strichartz} with \eqref{A7}, \eqref{A8}, 
and \eqref{A12}, 
we have
\begin{align}
\| \bv^\o - \bv_{N_k}^\o \|_{X([t_1, t_1+\dl])}
& \leq 
 C \| \bv^\o(t_1) - \bv_{N_k}^\o(t_1) \|_{\mathcal{\dot H}^1} 
+ 
\tfrac 12 
\| \bv^\o - \bv_{N_k}^\o \|_{L^3_t([t_1, t_1+ \dl]; L^6_x)}  \notag \\
& \hphantom{X|}
 + 
\tfrac 12 \| \bz^\o - \bz_{N_k}^\o \|_{L^3_t([t_1, t_1+ \dl]; L^6_x)}.
\label{A13}
\end{align}

\noi
Hence, by \eqref{A11} and \eqref{A13}, we have 
\begin{align}
\| \bv^\o - \bv_{N_k}^\o \|_{X([t_1, t_1+\dl])}
& \leq 
2 C  \| \bz^\o - \bz_{N_k}^\o \|_{L^3_{I_0} L^6_x}
 + 
 \| \bz^\o - \bz_{N_k}^\o \|_{L^3_t([t_1, t_1+ \dl]; L^6_x)}.
\label{A14}
\end{align}

\noi
Applying the continuity argument again
with  \eqref{A4}, \eqref{A7}, \eqref{A8}, and \eqref{A14}, 
it follows that there exists $K_1 \geq K_0$ such that 
\begin{align*}
\| \bv_{N_k}^\o \|_{L^3_{I_1} L^6_x}\leq 2 \g 
\qquad \text{and}\qquad 
\| \bv^\o - \bv_{N_k}^\o \|_{X({I_1})}
\leq  (2C+1) \| \bz^\o - \bz_{N_k}^\o \|_{L^3_T L^6_x}
\end{align*}

\noi
for all $k \geq K_1$.

By arguing inductively, we conclude that 
there exists $K_J\in \mathbb N$ such that 
\begin{align*}
\| \bv^\o - \bv_{N_k}^\o \|_{X([0, T])}
\leq  C_T  \| \bz^\o - \bz_{N_k}^\o \|_{L^3_T L^6_x}
< \frac{C_T}k
\end{align*}

\noi
for all $k \geq K_J$.
In particular, 
$v_{N_k}^\o  = \bv_{N_k}^\o|_{[0, T]\times \T^4}$ 
converges to 
$v^\o : = \bv^\o|_{[0, T]\times \T^4}$

\noi
in 
\begin{align}
 L^3_t([0, T];  L^6_x (\T^4))\cap C([0, T]; \dot H^1_x(\T^4)). 
\label{A16}
\end{align} 

\noi
Moreover,  $\partial_t v_{N_k}^\o$ 
converges to 
$\partial_t v^\o$
in $C([0, T]; L^2_x(\T^4))$. 
It follows from \eqref{A3},  \eqref{A16}, and the fact that
$\bv^\o$ satisfies \eqref{NLW4} on $[0, T] \times \R^4$,
that $v^\o$ is a distributional solution
to the following perturbed NLW on $\T^4$:
\begin{align*}
\begin{cases}
 \dt^2 v^\o - \Dl v^\o + (v^\o+z_{\text{per}}^\o)^3 = 0\\
 (v^\o, \dt v^\o) |_{t = 0} = (0, 0), 
\end{cases}
\qquad (t, x) \in [0, T]\times \T^4.
\end{align*}

\noi
Moreover, 
$v^\o$ satisfies the following Duhamel formulation:
\[ v^\o(t) = -  \int_0^t S_\text{per}(t-t') (v^\o(t') + z_\text{per}^\o(t'))^3 dt'\]

\noi
for $t \in [0, T]$.
This can be seen from the fact that 
$v^\o_{N_k}$ satisfies the corresponding Duhamel formulation for \eqref{NLW8}, 
the convergence of $v^\o_{N_k}$ to $v^\o$ in \eqref{A16},
and the convergence of $z_{\text{per}, {N_k}}$
to $z_{\text{per}}$ given by \eqref{A3} and \eqref{A4}.
Therefore, $u^\o :=  z^\o_\text{per} + v^\o $ is a solution to 
\eqref{NLW} on $[0, T]\times \T^4$
in the class \eqref{class1}.
This shows 
how Proposition \ref{PROP:GWP} follows from  Proposition \ref{PROP:GWP2}.

\begin{remark}\rm
In the above argument,  we only controlled the homogeneous $\dot H^1$-norm of
$v^\o$ for simplicity.
One can easily control 
 the nonhomogeneous $H^1$-norm of $v^\o$
by estimating the $L^2$-norm of $v^\o$
from the control on the $L^2$-norm of $\dt v^\o$ and Cauchy-Schwarz inequality (in time).
Since this is standard, we omit details.
\end{remark}

\section{On  uniqueness}\label{SEC:uniq}

We briefly discuss the issue on uniqueness mentioned in Remark \ref{REM:uniq}.
It follows from the proof of Theorem \ref{THM:GWP} that 
the set $\O_{(u_0, u_1)} $ 
can be written as $\O_{(u_0, u_1)} =\bigcup_{\eps>0}\O_\eps$
with $P(\O_\eps^c)<\eps$
such that
(i) there exists a global solution $u^\o$ to \eqref{NLW}
and (ii)  given any $T>0$, 
we have 
\begin{equation}
\|f^\omega\|_{L^{\frac{d+2}{d-2}}([0, T]; L^{\frac{2(d+2)}{d-2}}(\T^d))}
\leq C_1(T) < \infty,
\label{B1}
\end{equation}

\noi
for all $ \o \in \O_\eps$, where $f^\o := S_\textup{per}(\cdot)(u_0^\omega, u_1^\omega)$.
Now, we fix such $\o \in \O_\eps$ and suppress the dependence on $\o$ in the following.
Letting $v = u - f$, we see that $v$ is a global solution to 
the perturbed NLW  on $\T^d$:
\begin{align}
\begin{cases}
\pa_{t}^2v-\Delta v+F(v+f)=0\\
(v, \dt v)|_{t = 0} = (0, 0).
\end{cases}
\label{B}
\end{align}

Suppose that  $v_1, v_2 
\in X(\R)$
are two global solutions to \eqref{B},
where $X(\R)$ is as in \eqref{XX}.
Then, for each  $T > 0$, we have  
\begin{equation}
\|v_j\|_{L^{\frac{d+2}{d-2}}([0, T]; L^{\frac{2(d+2)}{d-2}}(\T^d))}
\leq C_2(T) < \infty,\quad j = 1, 2.
\label{B3}
\end{equation}

\noi
In view of \eqref{B1} and \eqref{B3}, 
we can write $[0, T] = \bigcup_{j = 0}^J I_j$
with  $I_j =  [t_j, t_{j+1}]$, $t_0 = 0 < t_1 < \cdots < t_J = T$,  and $J < \infty$ such that 
\begin{equation}
\|f^\omega\|_{L^{\frac{d+2}{d-2}}_{I_j}L_x^{\frac{2(d+2)}{d-2}}}
+ \sum_{j = 1}^2  \|v_j\|_{L^{\frac{d+2}{d-2}}_{I_j}L_x^{\frac{2(d+2)}{d-2}}}
\leq \g \ll 1.
\label{B4}
\end{equation}

Given a finite  interval $I$, let 
$X(I) = \{(w,\dt w): 
(w, \dt w) \in  C_I \mathcal{\dot H}^1_x(\T^d), \ 
 w\in
L^{\frac{d+2}{d-2}}_{I}L_x^{\frac{2(d+2)}{d-2}}(\T^d)
\}$.
Then, by a standard deterministic local-in-time analysis with Lemma \ref{LEM:Strichartz}
and \eqref{B4}, 
we obtain 
\begin{align*}
\| v_1 - v_2\|_{X(I_0) }
\leq C(\g)\| v_1 - v_2\|_{L^{\frac{d+2}{d-2}}_{I_0}L_x^{\frac{2(d+2)}{d-2}} }
\leq \frac 12 \| v_1 - v_2\|_{L^{\frac{d+2}{d-2}}_{I_0}L_x^{\frac{2(d+2)}{d-2}} }. 
\end{align*}

\noi
Therefore, we conclude that 
$ v_1 =  v_2$ on $I_0$.
In particular, we have $v_1(t_1) = v_2(t_1)$.
Thus,  we can iterate the above argument
and conclude that
$ v_1 =  v_2$ on $I_j$, $j = 1, 2, \dots, J$.
Namely, $v_1 = v_2$ on $[0, T]$.
Since the choice of $T$ was arbitrary, 
we conclude that $v_1 = v_2 $ on $[0, \infty)$.
Clearly, the same argument works for negative times.

\smallskip

\begin{ackno}\rm
T.O. was supported by the European Research Council (grant no.~637995 ``ProbDynDispEq'').
The authors would like to thank Pieter Blue 
for helpful discussions.
\end{ackno}



\end{document}